\documentclass[11pt,reqno]{amsart}
\usepackage{amsmath}
\usepackage{cases}
\usepackage{mathrsfs}
\usepackage{bbm}
\usepackage{amssymb}
\usepackage{amscd}
\usepackage{amsfonts,latexsym,amsmath,
amsthm,amsxtra,mathdots,amssymb,latexsym,mathabx}
\usepackage[all,cmtip]{xy}
\RequirePackage{amsmath} \RequirePackage{amssymb}
\usepackage{color}
\usepackage{colordvi}
\usepackage{multicol}
\usepackage{hyperref}
\usepackage{mathtools}
\usepackage[margin=1in]{geometry}
\usepackage{xcolor}

\hypersetup{
    colorlinks,
    linkcolor={red!50!black},
    citecolor={blue!100!black},
    urlcolor={blue!100!black}
}
\usepackage{cite}

\newcommand{\bea}{\begin{eqnarray}}
\newcommand{\eea}{\end{eqnarray}}
\newcommand{\bna}{\begin{eqnarray*}}
\newcommand{\ena}{\end{eqnarray*}}

\numberwithin{equation}{section}

\theoremstyle{plain}
\newtheorem{lemma}{Lemma}[section]
\newtheorem{theorem}[lemma]{Theorem}
\newtheorem{corollary}[lemma]{Corollary}
\newtheorem{proposition}[lemma]{Proposition}

\theoremstyle{definition}

\newtheorem{remark}{Remark}

\renewcommand{\Re}{\operatorname{Re}}
\renewcommand{\Im}{\operatorname{Im}}

\newcommand{\dd}{\mathrm{d}}

\newcommand{\sgn}{\operatorname{sgn}}

\renewcommand{\mod}{\operatorname{mod}\ }

\title[]
{A Zero-density estimate for $L$-functions associated with $\rm GL(3)$ Hecke--Maass cusp forms}

\author{Qingfeng Sun}
	\address{School of Mathematics and Statistics, Shandong University, Weihai\\Weihai, Shandong 264209, China}
	\email{qfsun@sdu.edu.cn}

    \author{Hui Wang}
  \address{School of Mathematics, Shandong University, Jinan, Shandong 250100, China}
    \email{wh0315@mail.sdu.edu.cn}

\subjclass[2020]{}


\thanks{Q. Sun was partially supported by the National Natural Science Foundation of China (Grant Nos.
12471005 and 12031008) and
the Natural Science Foundation of Shandong Province (Grant No. ZR2023MA003).}

\date{}

\begin{document}
\date{\today}

\begin{abstract}
In this paper, we establish an asymptotic formula for the twisted second moment of $L$-functions
associated with Hecke--Maass cusp forms for $\rm SL(3,\mathbb{Z})$, and
further deduce a weighted zero-density estimate for these
$L$-functions in the spectral aspect which may have important applications in other problems.
\end{abstract}


\maketitle
\setcounter{tocdepth}{1}

\section{Introduction} \label{sec:intr}
The zero-density estimates for automorphic $L$-functions play a significant role in many branches of analytic number theory.
For example, it is very effective in some applications to remove the Riemann Hypothesis.
Let $L(s)$ be an automorphic $L$-functions
which is absolutely convergent for $\Re(s)>1$,
has an analytic continuation to the whole complex plane, and satisfies a standard functional equation.
For $T>0$ and any fixed $\sigma\geq 0$, we define
\begin{equation*}
\begin{split}
N_{L}(\sigma; T)=\#\big\{\rho=1/2+\beta+i\gamma: L(\rho)=0,\,\sigma\leq\beta,\,|\gamma|\leq T\big\}.
\end{split}
\end{equation*}

For the Riemann zeta function $\zeta(s)$, the most classical result can be
tracked back to Ingham's famous work \cite{Ingham} on zero-density estimates,
where it was proved that
\begin{equation}\label{zero-density wrt zeta 1}
\begin{split}
N_{\zeta}(\sigma; T)\ll T^{3(1-\sigma)/(2-\sigma)}(\log T)^5.
\end{split}
\end{equation}
Later, number theorists such as Hal\'{a}sz, Montgomery, Huxley, Jutila, Heath--Brown
and Ivi\'{c} improved this result
(see, for example, \cite{Halasz, Montgomery, Huxley, Jutila, HB, Ivic}).
Recently, Guth and Maynard \cite{GM} proved that
\begin{equation*}
\begin{split}
N_{\zeta}(\sigma; T)\ll T^{15(1-\sigma)/(3+5\sigma)+\varepsilon}.
\end{split}
\end{equation*}
Combining this with Ingham's estimate \eqref{zero-density wrt zeta 1} when $\sigma\leq 7/10$, they further obtained (see \cite[Eq.\,(1.4)]{GM})
\begin{equation*}
\begin{split}
N_{\zeta}(\sigma; T)\ll T^{30(1-\sigma)/13+\varepsilon}.
\end{split}
\end{equation*}
Here the exponent $30/13$ improves on the previous exponent of $12/5$ due to Huxley \cite{Huxley}.

It is worth mentioning that Selberg \cite[Theorem 1]{Selberg1} made a major contribution
to this problem
by proving the uniform bound that for $0\leq\sigma\leq1/2$,
\begin{equation}\label{zero-density wrt zeta 2}
\begin{split}
N_{\zeta}(\sigma; T)\ll T^{1-\sigma/4}\log T.
\end{split}
\end{equation}
The crucial feature of \eqref{zero-density wrt zeta 2} is that the power of $\log T$ matches the true
order of the number of zeros of $\zeta(s)$ up to height $T$, so that the estimate is still useful even $\sigma$
is on the order of $1/\log T$. Later, Selberg \cite[Theorem 4]{Selberg} proved an analogous zero-density estimate in the family
of Dirichlet $L$-functions $L(s, \chi)$ corresponding to the primitive character $\chi$ modulo a large positive integer $q$.

For $\rm GL(2)$ $L$-functions, there have been fruitful results. In the case that $f$ is a holomorphic
cusp form of even integral weight $k>0$ and level $N$ with respect to
the Hecke congruence subgroup $\Gamma_0(N)$,
Luo \cite{Luo1} considered the case $N=1$ and first established
the analogous zero-density estimate as Selberg \cite{Selberg1} in $t$-aspect for $L(s,f)$,
\begin{equation*}
\begin{split}
N_{L}(\sigma; T)\ll_f T^{1-\sigma/72}\log T,
\end{split}
\end{equation*}
uniformly for $0\leq\sigma\leq 1/2$. Later, Kowalski--Michel \cite{KM} and Hough \cite{Hough}
extended Luo's result to the family of $L$-functions associated to
holomorphic forms $f$ of large prime level $N$ and large weight $k$, respectively.
In the case $f$ is a Hecke--Maass cusp form  Laplacian eigenvalue $1/4+t_f^2$ with $t_f\geq0$ for the full modular group, Liu--Streipel \cite{LS24} recently established
a weighted zero-density result in the spectral aspect by means of the twisted second moment
of $L$-functions $L(s,f)$ at $s=1/2+\sigma+i\tau$ with $\sigma>0$.

In this paper, we aim to establish the analogous zero-density estimate for $\rm GL(3)$
$L$-functions in the spectral aspect.
Let $\phi_j$ run over an orthogonal basis of Hecke--Maass cusp forms for $\rm SL(3,\mathbb{Z})$.
Each $\phi_j$ has spectral parameter $\nu_{j}=\big(\nu_{j,1},\nu_{j,2},\nu_{j,3}\big)$, the Langlands parameter
$\mu_{j}=\big(\mu_{j,1},\mu_{j,2},\mu_{j,3}\big)$, and the Hecke eigenvalues $A_j(m,n)$ with $A_j(1,1)=1$.
The $L$-function associated with $\phi_j$ is defined by
\[
L(s,\phi_j)=\sum_{n=1}^{\infty}\frac{A_{j}(1,n)}{n^{s}}
=\prod_{p}\left(1-\frac{A_j(1,p)}{p^s}
+\frac{A_j(p,1)}{p^{2s}}-\frac{1}{p^{3s}}\right)^{-1},
\quad\,\Re(s)>1.
\]

As in Blomer--Buttcane \cite{BB}, we consider the generic case,
i.e., away from the Weyl chamber walls and away from self-dual forms.
Let $\mu_0=(\mu_{0,1},\mu_{0,2},\mu_{0,3})$ and $\nu_0=(\nu_{0,1},\nu_{0,2},\nu_{0,3})$, and satisfy
the corresponding relations \eqref{eqn:m2n} and \eqref{eqn:n2m} below.
We also assume
\[
|\mu_{0,i}|\asymp|\nu_{0,i}|\asymp T:=\|\mu_0\|, \quad
1\leq i \leq 3,
\]
where
$$
\|\mu_0\|=\sqrt{|\mu_{0,1}|^2+|\mu_{0,2}|^2+|\mu_{0,3}|^2}.
$$
Throughout this paper, we let $M\!=\!T^\theta$ for any fixed $0<\theta<1$. Define a test function $h_{T,M}(\mu)$ (depending on $\mu_0$) for $\mu=(\mu_1,\mu_2,\mu_3)\in\mathbb{C}^3$ by
$$
h_{T,M}(\mu):=P(\mu)^2\bigg(\sum_{w\in\mathcal{W}}\psi\bigg(\frac{w(\mu)-\mu_0}{M}\bigg)\bigg)^2,
$$
where
$$
\psi(\mu)=\exp\left(\mu_{1}^2+\mu_{2}^2+\mu_{3}^2\right)
$$
and
$$
P(\mu)=\prod_{0\leq n\leq A} \prod_{k=1}^{3} \frac{\left(\nu_{k}-\frac{1}{3}(1+2n)\right)\left(\nu_{k}+\frac{1}{3}(1+2n)\right)}{|\nu_{0,k}|^2}
$$
for some fixed large $A>0$. Here
\[
\mathcal{W} := \left\{ I,\; w_2=\left(\begin{smallmatrix} 1 & & \\  & & 1 \\  &1&   \end{smallmatrix}\right),\;
w_3=\left(\begin{smallmatrix}  &1& \\  1&& \\  &&1   \end{smallmatrix}\right),\;
w_4=\left(\begin{smallmatrix}  &1& \\  &&1 \\  1&&   \end{smallmatrix}\right),\;
w_5=\left(\begin{smallmatrix}  &&1 \\  1&& \\  &1&   \end{smallmatrix}\right),\;
w_6=\left(\begin{smallmatrix}  &&1 \\  &1& \\  1&&   \end{smallmatrix}\right)
\right\}
\]
is the Weyl group of $\rm SL(3,\mathbb{R})$.
The function $h_{T,M}(\mu)$ has the localizing effect at
a ball of radius $M$ about $w(\mu_0)$ for each $w\in\mathcal{W}$,
and other nice properties stated in \S\,\ref{subsec:Kuznetsov trace formula} below.
Then with the normalizing factor
$$
\mathcal{N}_j:=\|\phi_j\|^2\prod_{k=1}^{3}\cos\left(\frac{3}{2}\pi\nu_{j,k}\right),
$$
our main theorem is the zero-density estimate of weighted version in the spectral aspect
for the $L$-functions associated with $\rm GL(3)$  Hecke--Maass forms.
\begin{theorem}\label{zero-density estimate}
For $\sigma>0$ and $H>0$, we let
\begin{equation*}
\begin{split}
N(\sigma, H; \phi_j):=\#\big\{\rho=1/2+\beta+i\gamma: L(\rho, \phi_j)=0, \beta>\sigma, |\gamma|<H\big\}
\end{split}
\end{equation*}
and
\begin{equation}\label{N(sigma,H)}
\begin{split}
N(\sigma, H):=\frac{1}{\mathcal{H}}\sum_{j}\frac{h_{T,M}(\mu_j)}{\mathcal{N}_j}N(\sigma, H; \phi_j).
\end{split}
\end{equation}
Let $2/\log T<\sigma<1/2$. For some sufficiently small $\delta,\,\theta_1>0$, we have
\begin{equation*}
\begin{split}
N(\sigma, H)\ll HT^{-\delta\sigma}\log T,
\end{split}
\end{equation*}
uniformly in $3/\log T<H<T^{\theta_1}$.
\end{theorem}
\begin{remark}
Theorem \ref{zero-density estimate} shows that there are only very few $L$-functions with zeros of imaginary part
less than $T^{\theta_1}$ and real part greater than $1/2+C/\log T$.
\end{remark}
\begin{remark}
Note that it seems difficult to prove an analogous result without harmonic weights (see, e.g., \cite{BZ}).
However, this weighted zero-density estimate is useful for certain problems,
such as removing the Generalized Riemann Hypothesis (GRH) in the work of Liu and Liu \cite{LL}.
And, the detailed proof of this application will be presented in the thesis of the second author.
\end{remark}
The main ingredient to prove Theorem \ref{zero-density estimate}
is the following asymptotic formula of the harmonic twisted
second moment of $L(s, \phi_j)$.
\begin{theorem}\label{MainThm}
Let $0<\sigma\leq1/2$, $|\tau|<T^{1/4}$, and $\ell_1$, $\ell_2<T^{1/3}$ be square-free integers.
Let the notations be as above and in \S\,\ref{subsec:Kuznetsov trace formula} below.
Then for any $\varepsilon>0$ we have
\begin{equation}\label{Aim1}
\begin{split}
&\sum_{j}\frac{h_{T,M}(\mu_j)}{\mathcal{N}_j}A_j(\ell_1,\ell_2)\left|L(1/2+\sigma+i\tau, \phi_j)\right|^2\\
=&\,\zeta(1+2\sigma)\frac{1}{(\ell_1\ell_2)^{1/2+\sigma}}\left(\frac{\ell_2}{\ell_1}\right)^{i\tau}\frac{1}{192\pi^5}
\int_{\Re(\mu)=0} h_{T,M}(\mu)\mathrm{d}_{\rm{spec}}\mu\\
&+\,\zeta(1-2\sigma)\frac{1}{(\ell_1\ell_2)^{1/2-\sigma}}\left(\frac{\ell_2}{\ell_1}\right)^{i\tau}\frac{1}{192\pi^5}
\int_{\Re(\mu)=0} h_{T,M}(\mu)\prod_{k=1}^{3} \left(-\frac{\mu^2_k}{4\pi^2}\right)^{-\sigma}\mathrm{d}_{\rm{spec}}\mu\\
&+\,O\big((\ell_1\ell_2)^{-1/2+\sigma}T^{9/4-6\sigma}M^2+T^{2-3\sigma+\varepsilon}M^2
+T^{83/42-41\sigma/21+\varepsilon}M^2(\ell_1\ell_2)^\vartheta\big),
\end{split}
\end{equation}
where $\vartheta$ is defined in \eqref{vartheta} below.
\end{theorem}
Note that both sides of \eqref{Aim1} are analytic in $\sigma$ near $0$. Letting $\sigma\rightarrow 0$, we have the following corollary.
\begin{corollary}
Let $|\tau|<T^{1/4}$ and $\ell_1$, $\ell_2<T^{1/3}$ be square-free integers. For any $\varepsilon>0$, we have
\begin{equation*}
\begin{split}
&\sum_{j}\frac{h_{T,M}(\mu_j)}{\mathcal{N}_j}A_j(\ell_1,\ell_2)\left|L(1/2+i\tau, \phi_j)\right|^2\\
=&\,\frac{1}{192\pi^5}\frac{1}{(\ell_1\ell_2)^{1/2}}\left(\frac{\ell_2}{\ell_1}\right)^{i\tau}
\int_{\Re(\mu)=0} h_{T,M}(\mu)\left(-\log(\ell_1\ell_2)+2\gamma+\sum_{k=1}^{3}
\log\left(-\frac{\mu^2_k}{4\pi^2}\right)\right)\mathrm{d}_{\rm{spec}}\mu\\
&+\,O\big((\ell_1\ell_2)^{-1/2}T^{9/4}M^2+T^{2+\varepsilon}M^2+
T^{83/42+\varepsilon}M^2(\ell_1\ell_2)^\vartheta\big).
\end{split}
\end{equation*}
\end{corollary}

The problem of moments of $L$-functions is another important topic in analytic number theory,
with many applications, such as in the subconvexity problem for $L$-functions
and for the non-vanishing of central $L$-values.
And to obtain positive-proportional non-vanishing results, one
typically employs the mollification method \`{a} la Selberg
and studies the mollified first and second moments of $L$-values at the central point $s=1/2$
 (see, for example \cite{BHS, HLX, Khan, KM, KM1, KMV, KMV1,  KMV2, Liu, Luo, Luo2, Sound, Xu}, etc.).

In this paper, we also investigate the asymptotic formula of the mollified
second moment of $L(s, \phi_j)$.
Define the mollifier $M(s,\phi_j)$ for $L(s,\phi_j)$ by
\begin{equation}\label{mollifier definition}
M(s, \phi_j):=\sum_{\ell\leq L^2}\frac{A_j(1,\ell)}{\ell^{s}} x_{\ell},
\end{equation}
where
\begin{align}\label{x_l definition}
x_{\ell}:=\mu(\ell)\frac{1}{2\pi i}\int_{(3)}\frac{(L^2/\ell)^u-(L/\ell)^u}{u^2}\frac{\mathrm{d} u}{\log L}
=\begin{cases}
\mu(\ell),\ \ &\ell\leq L,
\\ \noalign{\vskip 1mm}
\mu(\ell)\frac{\log(L^2/\ell)}{\log L},\ \ &L<\ell\leq L^2,
\\ \noalign{\vskip 1mm}
0,\ \ \ &\ell>L^2,
\end{cases}
\end{align}
here $\mu(\ell)$ is the M\"{o}bius function and $L=T^{\delta}$ for some suitably small constant $\delta>0$.
Then we have the following important proposition.
\begin{proposition}\label{second moment}
For sufficiently small $\delta,\,\theta_1>0$, we have
\begin{equation}\label{Aim2}
\frac{1}{\mathcal{H}}\sum_{j}\frac{h_{T,M}(\mu_j)}{\mathcal{N}_j}
\left|L(1/2+\sigma+i\tau, \phi_j)M(1/2+\sigma+i\tau, \phi_j)\right|^2\leq
1+O(T^{-2\sigma\delta}),
\end{equation}
uniformly for $1/\log T\leq\sigma\leq 1$ and $|\tau|<T^{\theta_1}$, and here
\bea\label{mathcal{H}}
\mathcal{H}:=\frac{1}{192\pi^5}\int_{\Re(\mu)=0}h_{T,M}(\mu)\mathrm{d}_{\rm{spec}}\mu.
\eea
Moreover, for all $\tau$, we have
\bna
L(3/2+i\tau, \phi_j)M(3/2+i\tau, \phi_j)=1+O(T^{-2\delta}).
\ena
\end{proposition}
\bigskip

\noindent\textbf{Structure of this paper.} This paper is organized as follows.
In \S\S\,\ref{sec:GL(3) Hecke--Maass L-functions}--\ref{sec:argument principle}, we
review some essential facts and useful tools, including $\rm GL(3)$ Hecke--Maass cusp forms and their $L$-functions,
Kuznetsov trace formula and the argument principle.
In \S\,\ref{sec:twisted2ndmoment} and \S\,\ref{sec:mollified2ndmoment},
we provide the asymptotic formulas for the harmonic twisted
and mollified second moment of $L(s, \phi_j)$, respectively, i.e., the detailed proofs of Theorem \ref{MainThm}
and Proposition \ref{second moment}. In \S\,\ref{sec:Proof of main theorem}, we first give an application of Kuznetsov trace formula and
then present the detailed proof of Theorem \ref{zero-density estimate}.

\bigskip

\noindent{\bf Notation.}
Through out the paper, $\varepsilon$ is an arbitrarily small positive number
and $B$ is a sufficiently large positive number which may not be the same at each occurrence.

As usual, $e(z) = \exp (2 \pi i z) = e^{2 \pi i z}$.
The symbol $k\sim K$ means $K< k\le 2K$.
We write $f = O(g)$ or $f \ll g$ to mean $|f| \leq \mathcal{C}g$ for some unspecified positive constant $\mathcal{C}$.
We denote $f \asymp g$ to mean that $c_1 g\leq|f|\leq c_2g$ for some positive constants $c_1$ and $c_2$.
The sign $\overline d$ denotes the multiplicative inverse of $d\,(\bmod\,c)$.
\section{$\rm GL(3)$ Hecke--Maass $L$-functions}
\label{sec:GL(3) Hecke--Maass L-functions}
In this section, we review the $L$-functions associated with
$\rm GL(3)$ Hecke--Maass cusp forms and provide the corresponding approximate functional equations.
\subsection{Hecke--Maass cusp forms and their $L$-functions}\label{subsec:cusp}
Let $G=\rm GL(3,\mathbb{R})$ with maximal compact subgroup $K=\rm O(3,\mathbb{R})$ and center $Z\cong\mathbb{R}^\times$.
Let $\mathbb{H}_3=G/(K\cdot Z)$ be the generalized upper half-plane.
For $0\leq c\leq \infty$, let
\begin{equation*}
  \Lambda_{c} = \Big\{ \mu=(\mu_1,\mu_2,\mu_3)\in\mathbb{C}^3 \; \Big| \; \mu_1+\mu_2+\mu_3=0,\ |\Re(\mu_k)|\leq c,\ k=1,2,3 \Big\},
\end{equation*}
and
\begin{equation*}
  \Lambda_{c}' = \Big\{ \mu\in\Lambda_c \; \Big| \; \{-\mu_1,-\mu_2,-\mu_3\}=\{\overline{\mu_1},\overline{\mu_2},\overline{\mu_3}\} \Big\}.
\end{equation*}
Here $\mu$ is the Langlands parameter of a Hecke--Maass form $\phi$ in $L^2(SL(3,\mathbb{Z})\backslash\mathbb{H})$.
We will simultaneously use $\mu$ and spectral parameters $\nu=(\nu_1,\nu_2,\nu_3)$ coordinates, defined by
\begin{equation}\label{eqn:n2m}
  \nu_1=\frac{1}{3}(\mu_1-\mu_2),\quad
  \nu_2=\frac{1}{3}(\mu_2-\mu_3),\quad
  \nu_3=\frac{1}{3}(\mu_3-\mu_1).
\end{equation}
And we have
\begin{equation}\label{eqn:m2n}
  \mu_1=2\nu_1+\nu_2, \quad
  \mu_2=\nu_2-\nu_1, \quad
  \mu_3=-\nu_1-2\nu_2.
\end{equation}

Let $\phi$ be a Hecke--Maass cusp form for $\!\rm SL(3,\mathbb{Z})$ with Fourier coefficients
$A_{\phi}(m,n)\in\mathbb{C}$ for $(m,n)\in \mathbb{N}^2$.
We have the following Hecke relation (see Goldfeld \cite[Theorem 6.4.11]{Goldfeld})
\bea\label{Hecke relation}
A_\phi(m,1)A_\phi(1,n)=\sum_{d|(m,n)}A_\phi\left(\frac{m}{d},\frac{n}{d}\right).
\eea
The standard $L$-function of $\phi$ is given by
\[
  L(s,\phi):=\sum_{n=1}^{\infty} \frac{A_\phi(1,n)}{n^s}, \quad \textrm{for }\Re(s)>1.
\]
The dual form of $\phi$ is denoted by $\widetilde{\phi}$ with the Langlands parameter
$(-\mu_1,-\mu_2,-\mu_3)$ and the coefficients $A_{\widetilde{\phi}}(m,n)=\overline{A_\phi(m,n)}=A_\phi(n,m)$.
Then $L(s,\phi)$ admits an analytic continuation to all $s\in \mathbb{C}$ and satisfies
the functional equation
\begin{equation*}\label{eqn:FE}
\Lambda(s,\phi)=L(s,\phi)\prod_{j=1}^{3}\Gamma_{\mathbb{R}}(s-\mu_j)
=L(1-s,\widetilde{\phi})\prod_{j=1}^{3}\Gamma_{\mathbb{R}}(1-s+\mu_j)=\Lambda(1-s,\widetilde{\phi}),
\end{equation*}
where $\Gamma_{\mathbb{R}}(s) = \pi^{-s/2}\Gamma(s/2)$.

By \cite{BB}, the Langlands parameter of $\phi$ lies in $\Lambda_{5/14}'\subset \Lambda_{5/14}$,
and the non-exceptional parameters lie in $\Lambda_{0}'=\Lambda_{0}$.
\subsection{Stirling's formula}\label{subsec:Stirling}
By the Stirling asymptotic formula (see \cite[\S\,8.4, Eq. (4.03)]{O}),
for $|\arg(z)|\leq \pi-\varepsilon$, $|z|\gg 1$ and any $\varepsilon>0$,
\bna
\log \Gamma(z)=\left(z-\frac{1}{2}\right)\log z-z+\frac{1}{2}\log(2\pi)
+\sum_{j=1}^{K_1}\frac{B_{2j}}{2j(2j-1)z^{2j-1}}+O_{K_1, \varepsilon}\left(\frac{1}{|z|^{2K_1+1}}\right),
\ena
where $B_j$ is $j$--th Bernoulli number. Thus for $z=\sigma_1+it$, $\sigma_1$ fixed and $|t|\geq 2$,
\bna\label{Stirling approximation}
\Gamma(\sigma_1+it)=\sqrt{2\pi}(it)^{\sigma_1-1/2}e^{-\pi|t|/2}\left(\frac{|t|}{e}\right)^{it}
\left(1+\sum_{j=1}^{K_2}\frac{c_j}{t^j}+O_{\sigma_1,K_2,\varepsilon}
\bigg(\frac{1}{|t|^{K_2+1}}\bigg)\right),
\ena
where the constant $c_j$ depends on $j,\sigma_1$ and $\varepsilon$.
\subsection{The approximate functional equation}\label{subsec:AFE}
Let $G(u)=(\cos\frac{\pi u}{A})^{-100A}$, where $A$ is a positive integer.
Then $G(u)$ is even and holomorphic inside the strip $-\frac{A}{2}<\Re(u)<\frac{A}{2}$.
For $\sigma>0$ and $\tau\in \mathbb{R}$,
we will use the following approximate functional equation of $\big|L\big(1/2+\sigma+i\tau,\phi_j\big)\big|^2$.
\begin{lemma}\label{lemma:AFE2}
Let $\phi_j$ be a $\rm GL(3)$ Maass form with Langlands parameters
$$\mu_j=(\mu_{j,1}, \mu_{j,2}, \mu_{j,3}) \in \Lambda_{5/14}'.$$
For $s=1/2+\sigma+i\tau$ with $0<\sigma\leq 1/2$ and $\tau\in \mathbb{R}$, we have
\begin{equation*}\label{AFE}
\begin{split}
\left|L\left(s,\phi_j\right)\right|^2=&\sum_{r}\frac{1}{r^{1+2\sigma}}
\sum_{m,n}\frac{A_j(m,n)}{(mn)^{1/2+\sigma}}\left(\frac{m}{n}\right)^{i\tau}W_{\mu_{j}}\left(s, r^2mn\right)\\
&+\pi^{6\sigma}\sum_{r}\frac{1}{r^{1-2\sigma}}
\sum_{m,n}\frac{A_j(m,n)}{(mn)^{1/2-\sigma}}\left(\frac{m}{n}\right)^{i\tau}
\tilde{W}_{\mu_{j}}\left(s, r^2mn\right),
\end{split}
\end{equation*}
where
\begin{equation}\label{W_j(s,y)}
\begin{split}
W_{\mu_{j}}(s, y):=\frac{1}{2\pi i}\int_{(3)} (\pi^3y)^{-u}
\prod_{k=1}^{3} \frac{\Gamma\left(\frac{s+u-\mu_{j,k}}{2}\right)}
{\Gamma\left(\frac{s-\mu_{j,k}}{2}\right)}\frac{\Gamma\left(\frac{\overline{s}+u+\mu_{j,k}}{2}\right)}
{\Gamma\left(\frac{\overline{s}+\mu_{j,k}}{2}\right)}
G(u)\frac{\mathrm{d} u}{u},
\end{split}
\end{equation}
and
\begin{equation}\label{W_j(s,y)1}
\begin{split}
\tilde{W}_{\mu_{j}}(s, y):= \frac{1}{2\pi i}\int_{(3)} (\pi^3y)^{-u}
\prod_{k=1}^{3}\frac{\Gamma\left(\frac{1-s+u+\mu_{j,k}}{2}\right)}
{\Gamma\left(\frac{\overline{s}+\mu_{j,k}}{2}\right)}\frac{\Gamma\left(\frac{1-\overline{s}+u-\mu_{j,k}}{2}\right)}
{\Gamma\left(\frac{s-\mu_{j,k}}{2}\right)}
G(u)\frac{\mathrm{d} u}{u}.
\end{split}
\end{equation}
\end{lemma}
\begin{proof}
See e.g., Iwaniec--Kowalski \cite[\S\,5.2]{IK}.
\end{proof}
Note that the sum in \eqref{N(sigma,H)} is essentially supported on the generic forms
which satisfy
\bea\label{condition1}
  |\mu_{j,k}|\asymp|\nu_{j,k}|\asymp T, \quad 1\leq k \leq 3.
\eea
So we assume that $\phi_j$ also satisfies the above relation.
The functions $W_{\mu_{j}}(s, y)$ and $\tilde{W}_{\mu_{j}}(s, y)$ satisfy the following properties.
\begin{lemma}
Assume the condition \eqref{condition1}, and let $s=1/2+\sigma+i\tau$ with $0<\sigma\leq 1/2$ and $|\tau|<T^{1/4}$.
For any $\varepsilon>0$, we have
\begin{equation}\label{W bound}
\begin{split}
y^iW^{(i)}_{\mu_{j}}(s, y),\quad y^i\tilde{W}^{(i)}_{\mu_{j}}(s, y)\ll_{i,B} \left(1+\frac{y}{T^3}\right)^{-B},
\quad\text{for any integer}\,\,i\geq 0,
\end{split}
\end{equation}
\begin{equation}\label{W truncate}
\begin{split}
W_{\mu_{j}}(s, y)=\frac{1}{2\pi i}\int_{\varepsilon-i\log T}^{\varepsilon+i\log T} (\pi^3y)^{-u}
\prod_{k=1}^{3} \frac{\Gamma\left(\frac{s+u-\mu_{j,k}}{2}\right)}
{\Gamma\left(\frac{s-\mu_{j,k}}{2}\right)}\frac{\Gamma\left(\frac{\overline{s}+u+\mu_{j,k}}{2}\right)}
{\Gamma\left(\frac{\overline{s}+\mu_{j,k}}{2}\right)}
G(u)\frac{\mathrm{d} u}{u}+O_{\varepsilon,B}(y^\varepsilon T^{-B}),
\end{split}
\end{equation}
\begin{equation}\label{W truncate1}
\begin{split}
\tilde{W}_{\mu_{j}}(s, y)=\frac{1}{2\pi i}\int_{\varepsilon-i\log T}^{\varepsilon+i\log T}(\pi^3y)^{-u}
\prod_{k=1}^{3}\frac{\Gamma\left(\frac{1-s+u+\mu_{j,k}}{2}\right)}
{\Gamma\left(\frac{\overline{s}+\mu_{j,k}}{2}\right)}\frac{\Gamma\left(\frac{1-\overline{s}+u-\mu_{j,k}}{2}\right)}
{\Gamma\left(\frac{s-\mu_{j,k}}{2}\right)}
G(u)\frac{\mathrm{d} u}{u}+O_{\varepsilon,B}(y^\varepsilon T^{-B}),
\end{split}
\end{equation}
and for $y\ll T^3$,
\begin{equation}\label{W asymptotic formula}
\begin{split}
W_{\mu_{j}}(s, y)=1+O_{B}\left(\left(\frac{y}{T^3}\right)^{B}\right),
\end{split}
\end{equation}
\begin{equation}\label{W asymptotic formula1}
\begin{split}
\tilde{W}_{\mu_{j}}(s, y)=\prod_{k=1}^{3}\frac{\Gamma\left(\frac{1-s+\mu_{j,k}}{2}\right)}
{\Gamma\left(\frac{\overline{s}+\mu_{j,k}}{2}\right)}\frac{\Gamma\left(\frac{1-\overline{s}-\mu_{j,k}}{2}\right)}
{\Gamma\left(\frac{s-\mu_{j,k}}{2}\right)}
+O_{B}\left(\left(\frac{y}{T^3}\right)^{B}\right),
\end{split}
\end{equation}
provided for any positive integer $B<A/2$.
\end{lemma}
\begin{proof}
Shifting the line of integration to $\Re(u)=B$ and using Stirling's formula, we can get \eqref{W bound}.
For \eqref{W truncate}, we move the line of integration in \eqref{W_j(s,y)} to $\Re(u)=\varepsilon$ without passing any poles, and
bound the integral for $|\Im(u)|\leq\log T$ using Stirling's formula and the exponential decay of $G(u)$.
The main term in \eqref{W asymptotic formula} is the residue from the pole at $u=0$ when we shift the integral line to $\Re(u)=-B$, while
the residues at the other poles are exponentially small in view of \eqref{condition1}.
Similarly, we can get \eqref{W truncate1} and \eqref{W asymptotic formula1}, respectively.
\end{proof}

By Stirling's formula, if $|u|\ll |z|^{1/2}$, then
\bea\label{Stirling application 1}
\frac{\Gamma(z+u)}{\Gamma(z)}=z^u\left(1+\sum_{n=1}^{N_1}\frac{P_{n}(u)}{z^{n}}+O\Big(\frac{(1+|u|)^{2N_1+2}}{|z|^{N_1+1}}\Big)\right),
\eea
for certain polynomials $P_{n}(u)$ of degree $2n$.
For $\mu_j \in \Lambda_{1/2}'$ and $\tau<T^{1/4}$, we have
\begin{align}\label{W new form}
W_{\mu_{j}}(1/2+\sigma+i\tau,y)&=\frac{1}{2\pi i} \int_{(3)}(\pi^3y)^{-u} \prod_{k=1}^{3}
\prod_\pm \Big(\frac{1/2+\sigma\mp i\tau\pm\mu_{j,k}}{2}\Big)^{u/2} \nonumber\\
&\cdot \left(1+\sum_{n=1}^{N_1}\frac{2^nP_{n}(u)}{(1/2+\sigma\mp i\tau\pm\mu_{j,k})^{n}}
+O\Big(\frac{(1+|u|)^{2N_1+2}}{|\mu_{j,k}|^{N_1+1}}\Big)\right)
G(u)\frac{\dd u}{u},
\end{align}
and
\begin{align*}\label{W new form1}
\tilde{W}_{\mu_j}(1/2+\sigma+i\tau,y)
&=\frac{1}{2\pi i} \int_{(3)}(\pi^3y)^{-u} \prod_{k=1}^{3}
\prod_\pm \Big(\frac{1/2+\sigma\mp i\tau\pm\mu_{j,k}}{2}\Big)^{-\sigma+u/2} \nonumber\\
&\cdot \left(1+\sum_{n=1}^{N_2}\frac{2^nP_{n}(u)}{(1/2+\sigma\mp i\tau\pm\mu_{j,k})^{n}}
+O\Big(\frac{(1+|u|)^{2N_2+2}}{|\mu_{j,k}|^{N_2+1}}\Big)\right)
G(u)\frac{\dd u}{u},
\end{align*}
for a different collection of polynomials $P_n(u)$.
\\ \
\section{The Kuznetsov trace formula}
In this section, we state the Kuznetsov trace formula for the particular test function $h_{T,M}$
described in \S\,\ref{sec:intr} above.
\subsection{The minimal Eisenstein series and its Fourier coefficients}\label{subsec:minimal Eisenstein}
Let
$$
U_3=\left\{\begin{pmatrix} 1&*&* \\ 0&1&* \\ 0&0&1 \end{pmatrix}\right\} \cap \rm SL(3,\mathbb{Z}).
$$
For $z\in \mathbb{H}_3$ and $\Re(\nu_1),\Re(\nu_2)$ sufficiently large,
we define the minimal Eisenstein series
\[
E(z;\mu):=\sum_{\gamma\in U_3\backslash  \rm SL(3,\mathbb{Z})} I_{\nu_1,\nu_2}(\gamma z),
\]
where
\[
I_{\nu_1,\nu_2}(z):=y_1^{1+\nu_1+2\nu_2}y_2^{1+2\nu_1+\nu_2},
\]
and
\begin{equation}\label{z}
z=\begin{pmatrix} y_1y_2&y_1x_2&x_3 \\ 0&y_1&x_1 \\ 0&0&1 \end{pmatrix}
=\begin{pmatrix} 1&x_2&x_3 \\ 0&1&x_1 \\ 0&0&1 \end{pmatrix}
\begin{pmatrix} y_1y_2&0&0 \\ 0&y_1&0 \\ 0&0&1 \end{pmatrix},
\end{equation}
with $y_1,y_2>0$ and $x_1,x_2,x_3\in\mathbb{R}$.
$E(z;\mu)$ has Langlands parameter $\mu$ and admits a meromorphic continuation in $\nu_1$ and $\nu_2$.
The Fourier coefficients $A_{\mu}(m_1,m_2)$ are defined by
(see Goldfeld \cite[Theorem 10.8.6]{Goldfeld})
\[
A_{\mu}(m,1) = \sum_{d_1d_2d_3=m} d_1^{\mu_1}d_2^{\mu_2}d_3^{\mu_3},
\]
and satisfy the symmetry and Hecke relation (see Goldfeld \cite[Theorem 6.4.11]{Goldfeld})
\begin{equation}\label{min Eisenstein coefficient property}
\begin{split}
A_{\mu}(m,1)&=\overline{A_{\mu}(1,m)}, \\
A_{\mu}(m_1,m_2)&=\sum_{d|(m_1,m_2)} \mu(d) A_{\mu}(m_1/d,1)A_{\mu}(1,m_2/d).
\end{split}
\end{equation}
Hence we have
\begin{equation*}\label{eqn:Amin<<}
  |A_{\mu}(m_1,m_2)| \ll_{\varepsilon} (m_1m_2)^{\varepsilon},
  \quad \textrm{if}\
  \mu\in (i\mathbb{R})^3.
\end{equation*}
The $L$-function associated with $E(z;\mu)$ is given by
\begin{equation}\label{min Eisenstein L-function}
\begin{split}
L(s, E(z;\mu))=\sum_{m\geq 1}\frac{A_{\mu}(1,m)}{m^s}=\zeta(s+\mu_1)\zeta(s+\mu_2)\zeta(s+\mu_3),
\end{split}
\end{equation}
for $\Re(s)>1$.
In order to state the Kuznetsov trace formula in the \S\,\ref{subsec:Kuznetsov trace formula},
we introduce the normalized factor for the minimal Eisenstein series is
$$
\mathcal{N}_\mu=\frac{1}{16}\prod_{k=1}^3|\zeta(1+3\nu_{k})|^2
$$
corresponding to the minimal Eisenstein series $E(z;\mu)$,
where $$\mu=(\mu_1,\mu_2,\mu_3).$$
Recall that (see \cite[Eq. (3.11.8)]{Titchmarsh}) we have
$$
\frac{1}{\zeta(1+it)}\ll \log (1+|t|),
$$
which implies that
\begin{align*}\label{N min}
\frac{1}{\mathcal{N}_\mu}\ll \prod_{k=1}^3\log^2(1+|\nu_k|).
\end{align*}
\subsection{The maximal Eisenstein series and its Fourier coefficients}
\label{subsec:maximal Eisenstein}
Let
$$
P_{2,1}=\left\{\begin{pmatrix} *&*&* \\ *&*&* \\ 0&0&* \end{pmatrix}\right\} \cap \rm SL(3,\mathbb{Z}).
$$
Let $\mu\in\mathbb{C}$ have sufficiently large real part,
and let $g: \rm SL(2,\mathbb{Z})\backslash\mathbb{H}_2\rightarrow \mathbb{C}$
be a Hecke--Maass cusp form with $\|g\|=1$,
Langlands parameter $\mu_g\in i\mathbb{R}$
and Hecke eigenvalue $\lambda_g(m)$.
The maximal Eisenstein series twisted by a Maass form $g$ is defined by
\[
  E(z;\mu;g) := \sum_{\gamma \in P_{2,1}\backslash \rm SL(3,\mathbb{Z})}\det(\gamma z)^{1/2+\mu} g(\mathfrak{m}_{P_{2,1}}(\gamma z)),
\]
where $z$ is defined as in \eqref{z}, and
\[
\mathfrak{m}_{P_{2,1}}: \mathbb{H}_3 \rightarrow \mathbb{H}_2,\quad
\left(\begin{matrix}
y_1y_2 & y_1x_2 & x_3 \\
& y_1 & x_1 \\
&   & 1
\end{matrix}\right) \mapsto
\left(\begin{matrix}
y_2 & x_2  \\
&   1
\end{matrix}\right)
\]
is the restriction to the upper left corner.
It has a meromorphic continuation in $\mu$.
The Fourier coefficients are determined by
\[
  B_{\mu,g}(m,1) = \sum_{d_1d_2=m}\lambda_g(d_1)d_1^{\mu}d_2^{-2\mu},
\]
and satisfy the symmetry and Hecke relation as \eqref{min Eisenstein coefficient property} above
(see Goldfeld \cite[Proposition 10.9.3 and Theorem 6.4.11]{Goldfeld})
\begin{equation}\label{max Eisenstein coefficient property}
\begin{split}
B_{\mu,g}(m,1)&=\overline{B_{\mu,g}(1,m)}, \\
B_{\mu,g}(m_1,m_2)&=\sum_{d|(m_1,m_2)} \mu(d) B_{\mu,g}(m_1/d,1)B_{\mu,g}(1,m_2/d).
\end{split}
\end{equation}
Recall that the current best known estimate towards the Ramanujan--Petersson conjecture for Fourier coefficients of $\rm GL(2)$ Maass cusp forms
is due to Kim--Sarnak \cite[Appendix 2]{Kim},
\bea\label{vartheta}
|\lambda_g(n)|\ll_{\varepsilon} n^{\vartheta+\varepsilon}, \quad n\in\mathbb{N},
\eea
here $\vartheta=7/64$. The Ramanujan--Petersson conjecture predicts $\vartheta=0$.
Hence we have
\begin{equation*}\label{eqn:Amax<<}
  |B_{\mu,g}(m_1,m_2)| \ll_\varepsilon (m_1m_2)^{\vartheta+\varepsilon},
  \quad \textrm{if}\
  \mu\in i\mathbb{R}.
\end{equation*}
The $L$-function associated with $E(z;\mu;g)$ is given by
\begin{equation}\label{max Eisenstein L-function}
\begin{split}
L(s, E(z;\mu;g))=\sum_{m\geq 1}\frac{B_{\mu,g}(1,m)}{m^s}=\zeta(s-2\mu)L(s+\mu, g),
\end{split}
\end{equation}
for sufficiently large $\Re(s)$.
Note that $L(s,g)$ is the Hecke--Maass $L$-function associated with $g$ with real Fourier coefficients $\lambda_g(n)$.
The complete $L$-function is defined by
\begin{equation}
\begin{split}
\Lambda(s, E(z;\mu;g))=\prod_{k=1}^3\Gamma_{\mathbb{R}}(s+\mu_k')L(s, E(z;\mu;g))=\Lambda(1-s, E(z;-\mu;g)),
\end{split}
\end{equation}
where
\begin{equation}
\begin{split}\label{definition mu'}
\mu_1'=\mu+\mu_g, \quad \mu_2'=\mu-\mu_g, \quad\mu_3'=-2\mu.
\end{split}
\end{equation}
We also introduce the normalizing factor
$$
\mathcal{N}_{\mu,g}=8L(1,\textup{Ad}^2g)|L(1+3\mu,g)|^2,
$$
where $L(s,\textup{Ad}^2g)$ is the adjoint square $L$-function of $g$.
And we have the lower bounds
$$
L(1,\textup{Ad}^2g)\gg (1+|\mu_g|)^{-\varepsilon},  \quad
L(1+it,g)\gg (1+|t|+|\mu_g|)^{-\varepsilon},
$$
which follow from \cite{HL, HR, JS} and \cite{GLS}.
Therefore, for $\mu\in i\mathbb{R}$, it follows that
\begin{align}\label{N max}
  \frac{1}{\mathcal{N}_{\mu,g}}\ll (1+|\mu|+|\mu_g|)^{\varepsilon}.
\end{align}

\subsection{The Kloosterman sums}
\label{subsec:Kloosterman sum}
For $n_1$, $n_2$, $m_1$, $m_2$, $D_1$, $D_2\in\mathbb{N}$, we need the relevant Kloosterman sums
\begin{align*}
  \tilde{S}(n_1,n_2,m_1;D_1,D_2) :=\sum_{\substack{C_1\,(\bmod D_1),\, C_2\,(\bmod D_2)\\ (C_1,D_1)=(C_2,D_2/D_1)=1}}
  e\left(n_2\frac{\bar{C_1}C_2}{D_1}+m_1\frac{\bar{C_2}}{D_2/D_1}+n_1\frac{C_1}{D_1}\right)
\end{align*}
for $D_1|D_2$, and
\begin{align*}
  &S(n_1,m_2,m_1,n_2;D_1,D_2)
  \\
  &:=\sum_{\substack{B_1,C_1\,(\bmod D_1)\\ B_2,C_2\,(\bmod D_2)
  \\ D_1C_2+B_1B_2+D_2C_1\equiv0\,(\bmod D_1D_2)\\ (B_j,C_j,D_j)=1}}
  e\left(\frac{n_1B_1+m_1(Y_1D_2-Z_1B_2)}{D_1}+\frac{m_2B_2+n_2(Y_2D_1-Z_2B_1)}{D_2}\right),
\end{align*}
where $Y_jB_j+Z_jC_j\equiv1\,(\bmod D_j)$ for $j=1,2$.
We have the standard (Weil-type) bounds
\begin{align}\label{Larsen bound}
\tilde{S}(n_1,n_2,m_1;D_1,D_2)\ll ((m_1, D_2/D_1)D_1^2, (n_1,n_2,D_1)D_2)(D_1D_2)^\varepsilon,
\end{align}
and
\begin{align*}\label{Weil' type bound}
S(n_1,m_2,m_1,n_2;D_1,D_2)\ll (D_1D_2)^{1/2+\varepsilon}((D_1, D_2)(m_1n_1,[D_1,D_2])(m_2n_2,[D_1,D_2]))^{1/2+\varepsilon}.
\end{align*}
The first bound is due to Larsen \cite[Appendix]{BFG}, and the second is due to Stevens \cite[Theorem 2]{Buttcane1}.

\subsection{The Kuznetsov trace formula}\label{subsec:Kuznetsov trace formula}

Define the spectral measure on the hyperplane $\mu_1+\mu_2+\mu_3=0$ by
\[
\mathrm{d}_{\text{spec}}\mu=\text{spec}(\mu)\mathrm{d}\mu,
\]
where
$$
\text{spec}(\mu):= \prod_{k=1}^{3}\left(3\nu_k\tan\left(\frac{3\pi}{2}\nu_k\right)\right)\quad\mbox{and}\quad
\mathrm{d}\mu=\mathrm{d}\mu_1\mathrm{d}\mu_2=\mathrm{d}\mu_2\mathrm{d}\mu_3=\mathrm{d}\mu_1\mathrm{d}\mu_3.
$$
Following \cite[Theorems 2 \& 3]{Buttcane},
we define the following integral kernels in terms of
Mellin--Barnes representations.
For $s \in \mathbb{C}$, $\mu \in \Lambda_{\infty}$, define the meromorphic function
\bea\label{tilde G definition}
\tilde{G}^{\pm}(s,\mu):=\frac{\pi^{-3s}}{12288\pi^{7/2}}
\Biggl(\prod_{k=1}^3\frac{\Gamma(\frac{1}{2}(s-\mu_k))}
{\Gamma(\frac{1}{2}(1-s+\mu_k))} \pm i
\prod_{k=1}^3\frac{\Gamma(\frac{1}{2}(1+s-\mu_k))}{\Gamma(\frac{1}{2}(2-s+\mu_k))}\Biggr),
\eea
and for $s=(s_1,s_2)\in \mathbb{C}^2$, $\mu\in\Lambda_{\infty}$, define the meromorphic function
$$
G(s,\mu):=\frac{1}{\Gamma(s_1+s_2)}\prod_{k=1}^3\Gamma(s_1-\mu_k)\Gamma(s_2+\mu_k).
$$
The latter is essentially the double Mellin transform
of the $\rm GL(3)$ Whittaker function.
We also define the following trigonometric functions
\begin{displaymath}
\begin{split}
&S^{++}(s,\mu):=\frac{1}{24\pi^2}\prod_{k=1}^3 \cos\left(\frac{3}{2} \pi \nu_k\right),\\
&S^{+-}(s,\mu):=-\frac{1}{32\pi^2}\frac{\cos(\frac{3}{2}\pi\nu_2)\sin(\pi(s_1-\mu_1))
\sin(\pi(s_2+\mu_2))\sin(\pi(s_2+\mu_3))}{\sin(\frac{3}{2}\pi\nu_1)\sin(\frac{3}{2}\pi\nu_3)\sin(\pi(s_1+s_2))},\\
&S^{-+}(s,\mu):=-\frac{1}{32\pi^2}\frac{\cos(\frac{3}{2}\pi\nu_1)\sin(\pi(s_1-\mu_1))
\sin(\pi(s_1-\mu_2))\sin(\pi(s_2+\mu_3))}{\sin(\frac{3}{2}\pi\nu_2)\sin(\frac{3}{2}\pi\nu_3)\sin(\pi(s_1+s_2))}, \\
&S^{--}(s,\mu):=\frac{1}{32\pi^2}\frac{\cos(\frac{3}{2}\pi\nu_3)\sin(\pi(s_1-\mu_2))
\sin(\pi(s_2+\mu_2))}{\sin(\frac{3}{2}\pi\nu_2)\sin(\frac{3}{2}\pi\nu_1)}.
\end{split}
\end{displaymath}
For  $y \in \Bbb{R} \setminus \{0\}$ with
$\sgn(y)=\epsilon$, 
let
\bea\label{K_w_4 definition}
K_{w_4}(y;\mu):=\int_{-i\infty}^{i\infty} | y|^{-u} \tilde{G}^{\epsilon}(u, \mu) \frac{\mathrm{d}u}{2\pi i}.
\eea
Here we quote the result of \cite[Lemma 4]{BB}.
\begin{lemma}\label{K_w_4}
For $y\in \mathbb{R}\backslash\{0\}$ and $\mu\in \Lambda_0$, we have
\begin{align}\label{K definition}
K_{w_4}(y;\mu)=&\,\frac{1}{3072\pi^5}\int_0^\infty J^{-}_{\mu_1-\mu_2}(2\sqrt{u})\left(\frac{\pi^3|y|}{u^{3/2}}\right)^{-\mu_3}
\exp\left(-\frac{2i\pi^3y}{u}\right)\frac{\mathrm{d}u}{u}\nonumber\\
&+\frac{1}{3072\pi^5}\int_0^\infty \tilde{K}_{\mu_1-\mu_2}(2\sqrt{u})\left(\frac{\pi^3|y|}{u^{3/2}}\right)^{-\mu_3}
\exp\left(\frac{2i\pi^3y}{u}\right)\frac{\mathrm{d}u}{u},
\end{align}
where
\begin{align*}
\tilde{K}_{it}(x)&=\int_{-\infty}^\infty\cos(2x\sinh v)\exp(itv)\mathrm{d}v,\\
J^{-}_{it}(x)&=\int_{-\infty}^\infty\cos(2x\cosh v)\exp(itv)\mathrm{d}v,
\end{align*}
for $t\in \mathbb{R}$, $x>0$.
\end{lemma}

\begin{remark}
As in \cite[\S\,5]{BB}, the integrals in \eqref{K definition} are not absolutely convergent at $0$, but
since $\exp(\pm 2i\pi^3y/u)$ is highly oscillating in a neighborhood of $u=0$, they exist in a Riemann sense,
and the portion $0<u<1$ can be made absolutely convergent after partial integration.
\end{remark}
\begin{remark}
For the Bessel functions, we have the uniform bounds (see \cite[Eq. (4.13)]{BB})
\begin{align}\label{J- bound}
\frac{\partial^k}{\partial x^k}\tilde{K}_{it}(x),\quad
\frac{\partial^k}{\partial x^k}J^{-}_{it}(x)\ll_k \left(1+\frac{|t|}{x}\right)^k,
\end{align}
for $|t|$, $x\geq 1$ and non-negative integer $k$.
\end{remark}

For $y = (y_1, y_2) \in (\mathbb{R}\setminus \{0\})^2$ with
$\sgn(y_1) = \epsilon_1$, $\sgn(y_2) = \epsilon_2$, let
\begin{displaymath}
\begin{split}
K^{\epsilon_1,\epsilon_2}_{w_6}(y; \mu)
:=&\int_{-i\infty}^{i\infty}\int_{-i\infty}^{i\infty}|4\pi^2 y_1|^{-s_1}|4\pi^2 y_2|^{-s_2}
G(s,\mu)S^{\epsilon_1,\epsilon_2}(s, \mu)\frac{\mathrm{d} s_1\mathrm{d}s_2}{(2\pi i)^2}.
\end{split}
\end{displaymath}
We generally follow the Barnes integral convention that
the contour should pass to the right of all of the poles
of the Gamma functions in the form $\Gamma(s_j + a)$
and to the left of all of the poles of the Gamma functions
\footnote{Such Gamma functions may occur through
the functional equation
$ (\sin(\pi(s_1 + s_2))\Gamma(s_1 + s_2) )^{-1}
=  \frac{1}{\pi} \Gamma(1 - s_1 - s_2)$.}
in the form $\Gamma(a - s_j)$.
Moreover, we choose the contour such that all integrals
are absolutely convergent, which can always be arranged
by shifting the unbounded part appropriately.

Now we can state the Kuznetsov trace formula in the version of
Buttcane \cite[Theorems 2--4]{Buttcane}.

\begin{lemma}\label{lemma: KTF}
Let $n_1$, $n_2$, $m_1$, $m_2 \in \mathbb{N}$ and
let $h$ be a function  that is holomorphic on
$\Lambda_{1/2 + \delta}$ for some $\delta > 0$,
symmetric under the Weyl group $\mathcal{W}$, of rapid decay when
$|\Im(\mu_j)| \rightarrow \infty$,  and satisfies
\begin{equation*}\label{eqn: zeros}
h(3\nu_j \pm 1)  = 0, \quad j = 1, 2, 3.
\end{equation*}
Then we have
\begin{displaymath}
\begin{split}
&\mathcal{C} + \mathcal{E}_{\min} + \mathcal{E}_{\max}
=\Delta+\Sigma_4+\Sigma_5+\Sigma_6,
\end{split}
\end{displaymath}
where
\begin{equation*}
\begin{split}
\mathcal{C}&:=\sum_{j} \frac{h(\mu_{j})}{\mathcal{N}_j} \overline{A_{j}(m_1, m_2)} A_{j}(n_1, n_2), \\
\mathcal{E}_{\min} &:=\frac{1}{24(2\pi i)^2} \iint_{\Re(\mu)=0} \frac{h(\mu)}{\mathcal{N}_{\mu}}
\overline{A_{\mu}(m_1, m_2)} A_{\mu}(n_1, n_2)\mathrm{d}\mu_1\mathrm{d}\mu_2,\\
\mathcal{E}_{\max} &:=\sum_{g} \frac{1}{2\pi i} \int_{\Re(\mu)=0}
\frac{h(\mu+\mu_g,\mu-\mu_g,-2\mu)}{\mathcal{N}_{\mu,g}}
\overline{B_{\mu,g}(m_1, m_2)} B_{\mu,g}(n_1, n_2)\mathrm{d}\mu,\\
\end{split}
\end{equation*}
and
\begin{equation*}
\begin{split}
\Delta&:=\delta_{n_1, m_1} \delta_{n_2, m_2}  \frac{1}{192\pi^5} \int_{\Re(\mu)= 0} h(\mu) \mathrm{d}_{\rm{spec}}\mu,\\
\Sigma_4&:=\sum_{\epsilon=\pm 1} \sum_{\substack{D_2 \mid D_1\\  m_2 D_1=n_1 D_2^2}}
\frac{ \tilde{S}(-\epsilon n_2, m_2, m_1; D_2, D_1)}{D_1D_2}
\Phi_{w_4}\!\left(\frac{\epsilon m_1m_2n_2}{D_1 D_2} \right),  \\
\Sigma_5&:=\sum_{\epsilon=\pm 1} \sum_{\substack{D_1 \mid D_2\\ m_1 D_2 = n_2 D_1^2}}
\frac{ \tilde{S}(\epsilon n_1, m_1, m_2; D_1, D_2) }{D_1D_2}\Phi_{w_5}\!\left( \frac{\epsilon n_1m_1m_2}{D_1 D_2}\right),\\
\Sigma_6&:=\sum_{\epsilon_1,\epsilon_2=\pm 1} \sum_{D_1,D_2}
\frac{S(\epsilon_2n_2,\epsilon_1n_1,m_1,m_2;D_1, D_2)}{D_1D_2}\Phi_{w_6}\!
\left(-\frac{\epsilon_2m_1n_2D_2}{D_1^2},-\frac{\epsilon_1 m_2n_1D_1}{ D_2^2}\right),
\end{split}
\end{equation*}
with
\begin{equation}\label{defPhi}
\begin{split}
& \Phi_{w_4}(y):=\int_{\Re(\mu)=0} h(\mu) K_{w_4}(y;\mu)\mathrm{d}_{\rm{spec}}\mu,\\
& \Phi_{w_5}(y):=\int_{\Re(\mu)=0} h(\mu) K_{w_4}(-y;-\mu)\mathrm{d}_{\rm{spec}}\mu,\\
& \Phi_{w_6}(y_1,y_2):=\int_{\Re(\mu)=0}h(\mu) K^{\sgn(y_1),\sgn(y_2)}_{w_6}((y_1, y_2);\mu)\mathrm{d}_{\rm{spec}}\mu.
\end{split}
\end{equation}
\end{lemma}
The function $h_{T,M}(\mu)$ localizes at a ball of radius $M$ about $w(\mu_0)$ for each $w\in\mathcal{W}$.
We have
\begin{equation}\label{property1 for h}
\mathscr{D}_k h_{T,M}(\mu)\ll_k M^{-k},
\end{equation}
for any differential operator $\mathscr{D}_k$ of order $k$,
which we use frequently when we integrate by parts,
and sufficiently many differentiations can save arbitrarily many powers of $T$.
Moreover, by trivial estimate, we have
\begin{equation}\label{property2 for h}
  \int_{\Re(\mu)=0}h_{T,M}(\mu)\mathrm{d}_{\rm{spec}}\mu\asymp T^3M^2.
\end{equation}

\subsection{The weight functions}\label{subsec:weight functions}
For the weight functions, we will need the following results
in Blomer--Buttcane \cite[Lemma 1, Lemmas 8 \& 9]{BB}.
\begin{lemma}\label{truncate 1st}
For some large enough constant $B>0$, we have
\[
\Phi_{w_4}(y) \ll |y|^{1/10}T^B,\quad
\Phi_{w_5}(y) \ll |y|^{1/10}T^B,\quad
\Phi_{w_6}(y_1,y_2) \ll |y_1y_2|^{3/5}T^B.
\]
\end{lemma}
\begin{remark}\label{remark 5}
As in \cite[\S\,3.6]{BB}, Lemma \ref{truncate 1st} shows more quantitatively that with our particular choice of $h$ we can
truncate the $D_1,\, D_2$-sums at $D_1,\,D_2\ll T^B$ for some sufficiently large $B$ at the cost of an error term $O(T^{-1000})$,
provided that $n_1,\,n_2,\,m_1,\,m_2\ll T^{10}$, say.
\end{remark}
\begin{lemma}\label{truncate 2nd}
(1) If $0<|y|\leq T^{3-\varepsilon}$, then for any constant $B>0$, we have
\[
\Phi_{w_4}(y) \ll_{\varepsilon,B} T^{-B}.
\]\\
(2) If $T^{3-\varepsilon}<|y|$, then
\[
|y|^{k}\Phi_{w_4}^{(k)}(y)\ll_{k,\varepsilon} T^{3+2\varepsilon}M^2(T+|y|^{1/3})^{k}.
\]
\end{lemma}

\begin{lemma}\label{truncate 3rd}
Let $\Upsilon:=\min\{|y_1|^{1/3}|y_2|^{1/6},|y_1|^{1/6}|y_2|^{1/3}\}$.\\
(1) If $\Upsilon\ll T^{1-\varepsilon}$, then we have
\[ \Phi_{w_6}(y_1,y_2) \ll T^{-B}. \]\\
(2) If $\Upsilon\gg T^{1-\varepsilon}$, then we have
\begin{equation*}
\begin{split}
|y_1|^{k_1}|y_2|^{k_2}&\frac{\partial^{k_1}}{\partial y_1^{k_1}}\frac{\partial^{k_2}}{\partial y_2^{k_2}}
\Phi_{w_6}(y_1, y_2)\\
&\ll_{k_1, k_2,\varepsilon} T^{3+2\varepsilon}M^2\big(T+|y_1|^{1/2}+|y_1|^{1/3}|y_2|^{1/6}\big)^{k_1}
\big(T+|y_2|^{1/2}+|y_2|^{1/3}|y_1|^{1/6}\big)^{k_2}.
\end{split}
\end{equation*}
\end{lemma}
\begin{remark}
In the proof of Lemmas \ref{truncate 1st}--\ref{truncate 3rd},
the only two properties of $h_{T,M}(\mu)$ used are \eqref{property1 for h} and \eqref{property2 for h}.
Here, we remark that $h_i(\mu)$ also satisfies these two inequalities
by using properties of $W_{\mu,N_1}$ or $\tilde{W}_{\mu,N_2}$  (see \S\,\ref{sec:twisted2ndmoment} below).
Thus, $\Phi_{i,w_4}(y)$, $\Phi_{i,w_5}(y)$ and $\Phi_{i,w_6}(y)$ defined as in \eqref{defPhi} by using
the test function $h_i(\mu)$, also satisfy the corresponding bounds
in Lemmas \ref{truncate 1st}--\ref{truncate 3rd},
since the difference that we use \eqref{property1 for h} and \eqref{property2 for h}
instead of \cite[(3.7) and (3.8)]{BB} has no influence.
\end{remark}
Moreover, we also need the following lemma, which gives some strong bounds on the $\Phi_w$ functions.
\begin{lemma}\label{truncate 4th}

(1) For $y_1, y_2\in (T^{-100}, T^{100})$, $\epsilon_1, \epsilon_2\in \{\pm 1\}$, we have
$\Phi_{w_6}(\epsilon_1y_1, \epsilon_2y_2) \ll T^{3+\varepsilon}.$
\vspace{1em}

\noindent(2) For $y\in (T^{-100}, T^{100})$, $\epsilon\in \{\pm 1\}$, we have
$\Phi_{w_4}(\epsilon y) \ll T^{3+\varepsilon}.$
\end{lemma}
\begin{proof}
The proof closely follows \cite[Proposition 4.1]{BZ}, which we give here for completeness.
For the completely supported test function $h_{T,M}$ described above, we define
\begin{equation*}
\begin{split}
\check{h}_{T,M}(v_1,v_2)=\frac{3}{T^3}\int_{\Re(\mu)=0}h_{T,M}(\mu)v_1^{-\nu_1}v_2^{-\mu_2}\mathrm{d}_{\rm{spec}}\mu.
\end{split}
\end{equation*}
The support of $h_{T,M}(\mu)$ allows us to remove the tangents from $\rm{spec}(\mu)$ at a negligible cost,
and integration by parts in $\check{h}_{T,M}(v_1,v_2)$ constrains $v_1,\,v_2$ to a region $v_i=1+O(M^{-1}T^{\varepsilon})$.
Then we have the trivial bound $\check{h}_{T,M}(v_1,v_2)\ll M^2$.

Now for $y_1, y_2\in (T^{-100}, T^{100})$, $\epsilon_1, \epsilon_2\in \{\pm 1\}$, \cite[Lemma 4.2]{BZ} (using the Weyl-invariance of $h_{T,M}(\mu)$)
implies
\begin{equation*}
\begin{split}
\Phi_{w_6}(\epsilon_1y_1, \epsilon_2y_2)&:=\frac{T^3}{8192\pi^4}\sum_{d\in\{0,1\}^2\atop \eta\in\{\pm 1\}^3}(\eta_1\eta_2\eta_3\epsilon_1)^{d_1}
\epsilon_2^{d_2}\int_0^\infty\int_0^\infty\check{h}_{T,M}\left(\frac{y_2z_1^{3/2}}{y_1},\frac{y_1}{y_2z_2^{3/2}}\right)\\
&\quad\times \sgn(1+\eta_1z_1)^{d_1+d_2}\sgn(1+\eta_2z_2)^{d_1+d_2}\sgn(1+\eta_3z_3)^{d_1+d_2}\\
&\quad\times \left(\frac{2}{\pi}K_0\left(4\pi \sqrt{|z_4|}\right)-(-1)^{d_3}
Y_0\left(4\pi \sqrt{|z_4|}\right)\right)\frac{\mathrm{d}z_1\mathrm{d}z_2}{z_1z_2},
\end{split}
\end{equation*}
where $K_0$ and $Y_0$ are the usual Bessel functions and
\begin{equation*}
\begin{split}
z_3&=\frac{y_1^2}{y_2^2z_1z_2},\quad z_4=y_2(1+\eta_1\sqrt{z_1})(1+\eta_2\sqrt{z_2})(1+\eta_3\sqrt{z_3}),\\
d_3&=d_1+d_2-2d_1d_2.
\end{split}
\end{equation*}
Then using the trivial bound $\ll \log(3+\sqrt{|z_4|^{-1}})$ for $K_0,\, Y_0$-Bessel functions and changing variables
$\frac{y_2z_1^{3/2}}{y_1}\mapsto z_1'$ and $\frac{y_1}{y_2z_2^{3/2}}\mapsto z_2'$, respectively, we imply directly
\bna
\Phi_{w_6}(\epsilon_1y_1, \epsilon_2y_2)\ll T^{3+\varepsilon}.
\ena
For $y\in (T^{-100}, T^{100})$, $\epsilon\in \{\pm 1\}$, \cite[Lemma 4.2]{BZ} again implies that
\begin{equation*}
\begin{split}
\Phi_{w_4}(\epsilon y)&:=\frac{T^3}{64\pi^{13/2}}\sum_{d\in\{0,1\}}(\epsilon i)^{1-d}
\int_0^\infty\int_0^\infty\check{h}_{T,M}\left(\frac{z_1^{3}}{y},\frac{z^3_1z^3_2}{y^2}\right)\\
&\quad\times \sin 2\pi\bigg(\frac{d}{2}+z_1\bigg)\sin 2\pi\bigg(\frac{d}{2}+z_2\bigg)
\sin 2\pi\bigg(\frac{d}{2}+\frac{y}{z_1z_2}\bigg)\frac{\mathrm{d}z_1\mathrm{d}z_2}{z_1z_2},
\end{split}
\end{equation*}
and trivially bounding the integral gives $\Phi_{w_4}(\epsilon y)\ll T^{3+\varepsilon}$.
\end{proof}
\section{The argument principle}\label{sec:argument principle}
In this section, we use the following version of the argument principle, introduced by Selberg \cite[Lemma 14]{Selberg} and Conrey--Soundararajan \cite[Lemma 2.1]{CS}.
\begin{lemma}\label{The argument principle}
Let $\omega$ be a holomorphic function, which is non-zero in some half-plane $\Re(s)\geq W$.
Let $\mathcal{B}$ be the rectangular box with vertices $W_0\pm iH$, $W_1\pm iH$ where $H>0$
and $W_0<W<W_1$. Then
\begin{equation*}
\begin{split}
4H \sum_{\beta+i\gamma\in \mathcal{B}\atop \omega(\beta+i\gamma)=0}&\cos\left(\frac{\pi \gamma}{2H}\right)
\sinh\left(\frac{\pi(\beta-W_0)}{2H} \right)\\
&=\int_{-H}^H\cos\left(\frac{\pi t}{2H}\right)\log|\omega(W_0+it)|\mathrm{d}t\\
&+\int_{W_0}^{W_1}\sinh\left(\frac{\pi(\alpha-W_0)}{2H} \right)\log|\omega(\alpha+iH)\omega(\alpha-iH)|\mathrm{d}\alpha\\
&-\Re\int_{-H}^H\cos\left(\frac{\pi (W_1-W_0+it)}{2iH}\right)\log\omega(W_1+it)\mathrm{d}t.
\end{split}
\end{equation*}
\end{lemma}
\begin{remark}
The special feature of this lemma, which allows it to hold uniformly even for small
$H\asymp 1/\log T$, is that only the real part of the logarithm appears in the portion of the integral within the critical strip,
so that this part can be bounded using the second moment estimate from Proposition \ref{second moment}.
\end{remark}
\section{The twisted second moment}\label{sec:twisted2ndmoment}
In this section, we devote to giving the \textbf{proof of Theorem \ref{MainThm}}.
Using the approximate functional equation in Lemma \ref{lemma:AFE2}, we can rewrite the twisted second moment as
\begin{equation*}
\begin{split}
&\sum_{j}\frac{h_{T,M}(\mu_j)}{\mathcal{N}_j}A_j(\ell_1,\ell_2)
\left|L(1/2+\sigma+i\tau, \phi_j)\right|^2\\
=&\sum_{r}\frac{1}{r^{1+2\sigma}}\sum_{m,n}\frac{1}{(mn)^{1/2+\sigma}}\left(\frac{m}{n}\right)^{i\tau}
\sum_{j}\frac{h_{T,M}(\mu_j)}{\mathcal{N}_j}A_j(\ell_1,\ell_2)
\overline{A_j(n,m)}W_{\mu_{j}}\left(\frac{1}{2}+\sigma+i\tau, r^2mn\right)\\
&+\pi^{6\sigma}\sum_{r}\frac{1}{r^{1-2\sigma}}
\sum_{m,n}\frac{1}{(mn)^{1/2-\sigma}}\left(\frac{m}{n}\right)^{i\tau}
\sum_{j}\frac{h_{T,M}(\mu_j)}{\mathcal{N}_j}A_j(\ell_1,\ell_2)\overline{A_j(n,m)}
\tilde{W}_{\mu_{j}}\left(\frac{1}{2}+\sigma+i\tau, r^2mn\right)\\
:=&\,\mathcal{M}_1(\ell_1, \ell_2)+\mathcal{M}_2(\ell_1, \ell_2).
\end{split}
\end{equation*}
In view of \eqref{W bound}, we may truncate the summations over $m$, $n$ and $r$
to $r^2mn\leq T^{3+\varepsilon}$ with a negligible error term.

Now we use the Kuznetsov trace formula (see Lemma \ref{lemma: KTF}) for $\mathcal{M}_i(\ell_1, \ell_2)$, $i=1,2$, respectively, with the test functions
\begin{align}\label{test h1}
h_1(\mu)=h_{T,M}(\mu)W_{\mu}\left(\frac{1}{2}+\sigma+i\tau, r^2mn\right),
\end{align}
and
\begin{align}\label{test h2}
h_2(\mu)=h_{T,M}(\mu)\tilde{W}_{\mu}\left(\frac{1}{2}+\sigma+i\tau, r^2mn\right).
\end{align}
It turns out we are led to estimate
\begin{equation}\label{use Ku 1}
\begin{split}
\mathcal{M}_1(\ell_1, \ell_2)=\sum_{r}\frac{1}{r^{1+2\sigma}}\sum_{m,n}\frac{1}{(mn)^{1/2+\sigma}}\left(\frac{m}{n}\right)^{i\tau}
\Big(\Delta^{(1)}+\Sigma_4^{(1)}+\Sigma_5^{(1)}+\Sigma_6^{(1)}
-\mathcal{E}_{\mathrm{max}}^{(1)}-\mathcal{E}_{\mathrm{min}}^{(1)}\Big),
\end{split}
\end{equation}
and
\begin{equation}\label{use Ku 2}
\begin{split}
\mathcal{M}_2(\ell_1, \ell_2)=\pi^{6\sigma}\sum_{r}\frac{1}{r^{1-2\sigma}}\sum_{m,n}\frac{1}{(mn)^{1/2-\sigma}}\left(\frac{m}{n}\right)^{i\tau}
\Big(\Delta^{(2)}+\Sigma_4^{(2)}+\Sigma_5^{(2)}+\Sigma_6^{(2)}
-\mathcal{E}_{\mathrm{max}}^{(2)}-\mathcal{E}_{\mathrm{min}}^{(2)}\Big),
\end{split}
\end{equation}
where for $i=1,2$,
\begin{displaymath}
\begin{split}
\Delta^{(i)} &:=\delta_{\ell_1,n}\delta_{\ell_2, m}
\frac{1}{192\pi^5} \int_{\Re(\mu)= 0} h_i(\mu) \mathrm{d}_{\rm{spec}}\mu,\\
\Sigma_{4}^{(i)}&:=\sum_{\epsilon=\pm 1} \sum_{\substack{D_2 \mid D_1\\  m D_1= \ell_1 D_2^2}}
\frac{\tilde{S}(-\epsilon \ell_2, m, n; D_2, D_1)}{D_1D_2}
\Phi_{i,w_4}\!\left(\frac{\epsilon nm\ell_2}{D_1 D_2} \right),\\
\Sigma_{5}^{(i)}&:=\sum_{\epsilon=\pm 1} \sum_{\substack{D_1 \mid D_2\\ n D_2 = \ell_2 D_1^2}}
\frac{\tilde{S}(\epsilon \ell_1, n, m; D_1, D_2) }{D_1D_2}
\Phi_{i,w_5}\!\left( \frac{\epsilon \ell_1nm}{D_1 D_2}\right),\\
\Sigma_6^{(i)}&:=\sum_{\epsilon_1, \epsilon_2 = \pm 1} \sum_{D_1,  D_2  }
\frac{S(\epsilon_2 \ell_2, \epsilon_1 \ell_1, n, m; D_1, D_2)}{D_1D_2}
\Phi_{i,w_6}\! \left(-\frac{\epsilon_2 n\ell_2D_2}{D_1^2},
-\frac{\epsilon_1 m\ell_1D_1}{ D_2^2}\right),
\end{split}
\end{displaymath}
with $\Phi_{i,w_4}(y)$, $\Phi_{i,w_5}(y)$ and $\Phi_{i,w_6}(y)$ defined as in \eqref{defPhi} by using
the test functions $h_i(\mu)$ given by \eqref{test h1} and \eqref{test h2}, respectively;
moreover, we have
\begin{equation*}
\begin{split}
\mathcal{E}_{\max}^{(i)}&:=\sum_{g} \frac{1}{2\pi i} \int\limits_{\Re(\mu)=0}
\frac{h_i(\mu+\mu_g,\mu-\mu_g,-2\mu)}{\mathcal{N}_{\mu,g}}
\overline{B_{\mu,g}(n, m)} B_{\mu,g}(\ell_1, \ell_2) \dd\mu, \\
\mathcal{E}_{\min}^{(i)}&:=\frac{1}{24(2\pi i)^2} \iint\limits_{\Re(\mu)=0} \frac{h_i(\mu)}{\mathcal{N}_{\mu}}
\overline{A_{\mu}(n, m)} A_{\mu}(\ell_1, \ell_2) \dd\mu_1 \dd\mu_2.
\end{split}
\end{equation*}
\subsection{The diagonal terms $\Delta^{(1)}$ and $\Delta^{(2)}$}

Note that we have $\ell_1=n$ and $\ell_2=m$.
Thus we infer that the diagonal terms in \eqref{use Ku 1} and \eqref{use Ku 2} are
\begin{align}\label{diagonal}
\frac{1}{192\pi^5}\sum_{r}\frac{1}{r^{1+2\sigma}}\frac{1}{(\ell_1\ell_2)^{1/2+\sigma}}\left(\frac{\ell_2}{\ell_1}\right)^{i\tau}
\int_{\Re(\mu)=0} h_{T,M}(\mu)W_{\mu}\left(1/2+\sigma+i\tau, r^2\ell_1\ell_2\right) \mathrm{d}_{\rm{spec}}\mu,
\end{align}
and
\begin{align}\label{diagonal1}
\frac{\pi^{6\sigma}}{192\pi^5}\sum_{r}\frac{1}{r^{1-2\sigma}}\frac{1}{(\ell_1\ell_2)^{1/2-\sigma}}\left(\frac{\ell_2}{\ell_1}\right)^{i\tau}
\int_{\Re(\mu)=0} h_{T,M}(\mu)\tilde{W}_{\mu}\left(1/2+\sigma+i\tau, r^2\ell_1\ell_2\right) \mathrm{d}_{\rm{spec}}\mu.
\end{align}
Using the definitions in \eqref{W_j(s,y)} and \eqref{W_j(s,y)1} of $W_{\mu}(s,y)$ and $\tilde{W}_{\mu}(s,y)$, respectively,
we have \eqref{diagonal} and \eqref{diagonal1} equal to
\begin{align}\label{Insert W}
\begin{split}
\frac{1}{192\pi^5}\frac{1}{(\ell_1\ell_2)^{1/2+\sigma}}\left(\frac{\ell_2}{\ell_1}\right)^{i\tau}
&\int_{\Re(\mu)=0} h_{T,M}(\mu)\frac{1}{2\pi i}\int_{(3)}\zeta(1+2\sigma+2u) (\pi^3\ell_1\ell_2)^{-u}\\
&\cdot\prod_{k=1}^{3} \frac{\Gamma\left(\frac{1/2+\sigma+i\tau+u-\mu_{k}}{2}\right)}
{\Gamma\left(\frac{1/2+\sigma+i\tau-\mu_{k}}{2}\right)}\frac{\Gamma\left(\frac{1/2+\sigma-i\tau+u+\mu_{k}}{2}\right)}
{\Gamma\left(\frac{1/2+\sigma-i\tau+\mu_{k}}{2}\right)}
G(u)\frac{\mathrm{d} u}{u}\mathrm{d}_{\rm{spec}}\mu,
\end{split}
\end{align}
and
\begin{align}\label{Insert W1}
\begin{split}
\frac{\pi^{6\sigma}}{192\pi^5}\frac{1}{(\ell_1\ell_2)^{1/2-\sigma}}\left(\frac{\ell_2}{\ell_1}\right)^{i\tau}
&\int_{\Re(\mu)=0} h_{T,M}(\mu)\frac{1}{2\pi i}\int_{(3)}\zeta(1-2\sigma+2v) (\pi^3\ell_1\ell_2)^{-v}\\
&\cdot\prod_{k=1}^{3} \frac{\Gamma\left(\frac{1/2-\sigma+i\tau+v-\mu_{k}}{2}\right)}
{\Gamma\left(\frac{1/2+\sigma+i\tau-\mu_{k}}{2}\right)}\frac{\Gamma\left(\frac{1/2-\sigma-i\tau+v+\mu_{k}}{2}\right)}
{\Gamma\left(\frac{1/2+\sigma-i\tau+\mu_{k}}{2}\right)}
G(v)\frac{\mathrm{d}v}{v}\mathrm{d}_{\rm{spec}}\mu.
\end{split}
\end{align}
We shift the $u$-integral in \eqref{Insert W} to $\Re(u)=-B$ for a large $0<B<A/2$.
The main contributions are coming from the residues at two simple poles at $u=0$ and $u=-\sigma$.
Hence, by the residue theorem, we infer that
\begin{align*}
\begin{split}
&\zeta(1+2\sigma)\frac{1}{(\ell_1\ell_2)^{1/2+\sigma}}\left(\frac{\ell_2}{\ell_1}\right)^{i\tau}\frac{1}{192\pi^5}
\int_{\Re(\mu)=0} h_{T,M}(\mu)\mathrm{d}_{\rm{spec}}\mu\\
&-\frac{\pi^{3\sigma}}{384\pi^5\sigma}\frac{G(\sigma)}{(\ell_1\ell_2)^{1/2}}\left(\frac{\ell_2}{\ell_1}\right)^{i\tau}
\int_{\Re(\mu)=0} h_{T,M}(\mu)\prod_{k=1}^{3} \frac{\Gamma\left(\frac{1/2+i\tau-\mu_{k}}{2}\right)}
{\Gamma\left(\frac{1/2+\sigma+i\tau-\mu_{k}}{2}\right)}\frac{\Gamma\left(\frac{1/2-i\tau+\mu_{k}}{2}\right)}
{\Gamma\left(\frac{1/2+\sigma-i\tau+\mu_{k}}{2}\right)}\mathrm{d}_{\rm{spec}}\mu.
\end{split}
\end{align*}
Similarly, we shift the $v$-integral in \eqref{Insert W1} to $\Re(v)=-B$ for a large $0<B<A/2$.
The main contributions are coming from the residues at two simple poles at $v=0$ and $v=\sigma$, which are
\begin{align*}
\begin{split}
&\frac{\pi^{6\sigma}}{192\pi^5}\frac{\zeta(1-2\sigma)}{(\ell_1\ell_2)^{1/2-\sigma}}\left(\frac{\ell_2}{\ell_1}\right)^{i\tau}
\int_{\Re(\mu)=0} h_{T,M}(\mu)\prod_{k=1}^{3} \frac{\Gamma\left(\frac{1/2-\sigma+i\tau-\mu_{k}}{2}\right)}
{\Gamma\left(\frac{1/2+\sigma+i\tau-\mu_{k}}{2}\right)}\frac{\Gamma\left(\frac{1/2-\sigma-i\tau+\mu_{k}}{2}\right)}
{\Gamma\left(\frac{1/2+\sigma-i\tau+\mu_{k}}{2}\right)}\mathrm{d}_{\rm{spec}}\mu\\
&+\frac{\pi^{3\sigma}}{384\pi^5\sigma}\frac{G(\sigma)}{(\ell_1\ell_2)^{1/2}}\left(\frac{\ell_2}{\ell_1}\right)^{i\tau}
\int_{\Re(\mu)=0} h_{T,M}(\mu)\prod_{k=1}^{3} \frac{\Gamma\left(\frac{1/2+i\tau-\mu_{k}}{2}\right)}
{\Gamma\left(\frac{1/2+\sigma+i\tau-\mu_{k}}{2}\right)}\frac{\Gamma\left(\frac{1/2-i\tau+\mu_{k}}{2}\right)}
{\Gamma\left(\frac{1/2+\sigma-i\tau+\mu_{k}}{2}\right)}\mathrm{d}_{\rm{spec}}\mu.
\end{split}
\end{align*}

Note that the contribution from the residue at $u=-\sigma$ from \eqref{Insert W} and the residue at $v=\sigma$ from \eqref{Insert W1}
are canceled. By Stirling's formula and the exponentially decay of $G$, the contributions from the residues at the other places and the integral at $\Re(u)=\Re(v)=-B$ are negligible.
Moreover, using \eqref{Stirling application 1} and the Taylor expansion
\bna
(1+x)^\alpha=1+\alpha x+\frac{\alpha(\alpha-1)}{2!}x^2+\frac{\alpha(\alpha-1)(\alpha-2)}{3!}x^3+o(x^3),
\ena
we have
\bna
\prod_{k=1}^{3} \frac{\Gamma\left(\frac{1/2-\sigma+i\tau-\mu_{k}}{2}\right)}
{\Gamma\left(\frac{1/2+\sigma+i\tau-\mu_{k}}{2}\right)}\frac{\Gamma\left(\frac{1/2-\sigma-i\tau+\mu_{k}}{2}\right)}
{\Gamma\left(\frac{1/2+\sigma-i\tau+\mu_{k}}{2}\right)}
=\left(-\frac{\mu_1^2\mu_2^2\mu_3^2}{64}\right)^{-\sigma}+O\left(T^{-6\sigma-3/4}\right),
\ena
provided for $|\tau|\ll T^{1/4}$.
Thus, we obtain the main terms of Theorem \ref{MainThm}
and the first error term $O((\ell_1\ell_2)^{-1/2+\sigma}T^{9/4-6\sigma}M^2)$ by mean of \eqref{property2 for h}.

\subsection{The terms $\Sigma_{4}^{(i)}$ and $\Sigma_{5}^{(i)}$ for $i=1,2$}
In this subsection, we will focus on the contribution from the $w_4$-term in $\Sigma_{4}^{(1)}$,
that is,
\begin{align*}
\sum_{\epsilon=\pm1}\sum_{r}\frac{1}{r^{1+2\sigma}}\sum_{m,n}
\frac{1}{(mn)^{1/2+\sigma}}\left(\frac{m}{n}\right)^{i\tau}
\sum_{\substack{D_2 \mid D_1\\  m D_1= \ell_1 D_2^2}}
\frac{\tilde{S}(-\epsilon \ell_2, m, n; D_2, D_1)}{D_1D_2}
\Phi_{1,w_4}\!\left(\frac{\epsilon nm\ell_2}{D_1 D_2} \right),
\end{align*}
since the contributions from the $w_4$-term in $\Sigma_{4}^{(2)}$ and
the $w_5$-term in $\Sigma_{5}^{(i)}$ ($i=1,2$) are very similar.
Inserting a smooth partition of unity into $m$, $n$-sums, we are led to estimate
\begin{align}\label{w4 contribution}
\sum_{\epsilon=\pm1}\sum_{r}\frac{1}{r^{1+2\sigma}}\sum_{m,n}
\frac{W_1(\frac{m}{\mathcal{M}})W_2(\frac{n}{\mathcal{N}})}{(mn)^{1/2+\sigma}}\left(\frac{m}{n}\right)^{i\tau}
\sum_{\delta,D \atop m\delta=\ell_1 D}\frac{\tilde{S}(-\epsilon \ell_2, m, n; D, D\delta)}{D^2\delta}
\Phi_{1,w_4}\!\left(\frac{\epsilon nm\ell_2}{D^2\delta} \right),
\end{align}
where $\mathcal{M}, \mathcal{N}\gg1$, $\mathcal{M}\mathcal{N}\ll T^{3+\varepsilon}$ (by \eqref{W bound})
and $W_i(x)$ ($i=1,2$) are compactly supported in $[1,2]$ and satisfy $x^jW_i^{(j)}(x)\ll_j 1$, for any integer $j\geq0$.
Using the definitions \eqref{defPhi}, \eqref{test h1} and \eqref{W_j(s,y)} of $\Phi_{w_4}\!\left(y\right)$, $h_1(\mu)$ and $W_{\mu}(s, y)$, respectively,
we get that $\Phi_{1,w_4}\!\left(y\right)$ is given by
\begin{equation*}
\begin{split}
\int_{\Re(\mu)=0} h_{T,M}(\mu)\frac{1}{2\pi i}\int_{(3)} (\pi^3r^2mn)^{-u}
\prod_{k=1}^{3}\prod_{\pm} \frac{\Gamma\left(\frac{\frac{1}{2}+\sigma\pm i\tau+u\mp \mu_{k}}{2}\right)}
{\Gamma\left(\frac{\frac{1}{2}+\sigma\pm i\tau\mp \mu_{k}}{2}\right)}
G(u)\frac{\mathrm{d} u}{u}K_{w_4}(y;\mu)\mathrm{d}_{\rm{spec}}\mu.
\end{split}
\end{equation*}
Now we consider the new resulting $\mu$-integral
\begin{equation*}
\begin{split}
\int_{\Re(\mu)=0} h_{T,M}(\mu)
\prod_{k=1}^{3}\frac{\Gamma\left(\frac{\frac{1}{2}+\sigma+ i\tau+u- \mu_{k}}{2}\right)}
{\Gamma\left(\frac{\frac{1}{2}+\sigma+ i\tau-\mu_{k}}{2}\right)}
\frac{\Gamma\left(\frac{\frac{1}{2}+\sigma- i\tau+u+\mu_{k}}{2}\right)}
{\Gamma\left(\frac{\frac{1}{2}+\sigma- i\tau+ \mu_{k}}{2}\right)}
K_{w_4}(y;\mu)\mathrm{d}_{\rm{spec}}\mu,
\end{split}
\end{equation*}
with $|\Im(u)|\ll T^\varepsilon$, since the exponentially decay of $G(u)$.
With the help of \eqref{Stirling application 1}, we have
\begin{equation*}
\begin{split}
\prod_{k=1}^{3}\frac{\Gamma\left(\frac{\frac{1}{2}+\sigma+ i\tau+u- \mu_{k}}{2}\right)}
{\Gamma\left(\frac{\frac{1}{2}+\sigma+ i\tau-\mu_{k}}{2}\right)}
\frac{\Gamma\left(\frac{\frac{1}{2}+\sigma- i\tau+u+\mu_{k}}{2}\right)}
{\Gamma\left(\frac{\frac{1}{2}+\sigma- i\tau+ \mu_{k}}{2}\right)}\ll T^{O(1)}.
\end{split}
\end{equation*}
Then by applying the similar argument as Blomer--Buttcane \cite[Lemma 1]{BB}, we can truncate the $D$, $\delta$-sums
at some $T^{B_1}$ for some large ${B_1}$ at the cost of a negligible error.
In other words, we can use Lemma \ref{truncate 2nd} with $\Phi_{1,w_4}\!\left(y\right)$ replacing $\Phi_{w_4}\!\left(y\right)$.
Let
\begin{equation}
\begin{split}
W_{\mu,N_1}(1/2+\sigma+i\tau,y)=\frac{1}{2\pi i} \int_{(3)}(\pi^3y)^{-u}& \prod_{k=1}^{3}
\prod_\pm \Big(\frac{1/2+\sigma\mp i\tau\pm\mu_{k}}{2}\Big)^{u/2} \nonumber\\
&\cdot \left(1+\sum_{n=1}^{N_1}\frac{2^nP_{n}(u)}{(1/2+\sigma\mp i\tau\pm\mu_{k})^{n}}\right)
G(u)\frac{\dd u}{u}.
\end{split}
\end{equation}
Then with the help of \eqref{W new form},
we can replace $W_{\mu}$ by $W_{\mu,N_1}$ in \eqref{w4 contribution} with a negligible error
if we choose $N_1$ to be large enough.

Denote $\Phi^\star_{1,w_4}\!\left(y\right)$ as in \eqref{defPhi} with
$h_{T,M}(\mu)W_{\mu,N_1}\left(\frac{1}{2}+\sigma+i\tau, r^2mn\right)$ replacing $h_1(\mu)$.
Then, in \eqref{w4 contribution}, we can replace $\Phi_{1,w_4}\!\left(y\right)$ by $\Phi^\star_{1,w_4}\!\left(y\right)$
with a negligible error term if we choose $N_1$ to be large enough. Hence, together with the above truncation,
we only need to consider the contribution from $\Phi^\star_{1,w_4}\!\left(y\right)$, where $|y|>T^{-B_2}$ (by Lemma \ref{truncate 1st}).

Note that in the original proof of Lemma \ref{truncate 2nd} (see \cite[Lemma 8]{BB}),
the only two properties of $h_{T,M}(\mu)$ which are used are \eqref{property1 for h} and \eqref{property2 for h}.
Obviously, $h_{T,M}(\mu)W_{\mu,N_1}\left(\frac{1}{2}+\sigma+i\tau, r^2mn\right)$ also satisfies these two properties.
So we have
\begin{align*}
\Phi^\star_{1,w_4}\!\left(y\right)\ll_{\varepsilon,B_2}T^{-B_2},
\end{align*}
provided for $T^{-B_2}\leq |y|\leq T^{3-\varepsilon}$.
Then by Lemma \ref{truncate 2nd}, we can truncate the $D$, $\delta$-sums, again with a negligible error, at
\begin{align}\label{eqn:condition1}
\frac{mn\ell_2}{D^2\delta}\geq T^{3-\varepsilon},
\end{align}
or in other words,
$$
D^2\delta\leq \frac{\mathcal{M}\mathcal{N}\ell_2}{T^{3-\varepsilon}}\ll \ell_2T^\varepsilon.
$$
Note that we have $m\delta=\ell_1D$ now, which implies $\mathcal{M}$ is small. That is,
$$
\mathcal{M}\leq \ell_1D/\delta
\ll \ell_1\ell_2^{1/2}T^\varepsilon/\delta^{3/2}.
$$
And by \eqref{eqn:condition1} we have
$$
\mathcal{N}\gg T^{3-\varepsilon}D^2\delta/(m\ell_2)
=T^{3-\varepsilon}D\delta^2/(\ell_1\ell_2)\gg T^{3-\varepsilon}/(\ell_1\ell_2).
$$
Therefore, by the above evaluation and the definition of $\Phi^\star_{1,w_4}\!\left(y\right)$, we only need to estimate
\begin{align}\label{Phi star_1,w_4}
&\sum_{\epsilon=\pm1}\sum_{r}\frac{1}{r^{1+2\sigma}}\sum_{m,n}
\frac{W_1(\frac{m}{\mathcal{M}})W_2(\frac{n}{\mathcal{N}})}{(mn)^{1/2+\sigma}}\left(\frac{m}{n}\right)^{i\tau}
\sum_{\delta,D \atop m\delta=\ell_1 D}\frac{\tilde{S}(-\epsilon \ell_2, m, n; D, D\delta)}{D^2\delta}\nonumber\\
&\times\int_{\Re(\mu)=0} h_{T,M}(\mu)W_{\mu,N_1}\left(\frac{1}{2}+\sigma+i\tau, r^2mn\right)
K_{w_4}\left(\frac{\epsilon nm\ell_2}{D^2\delta};\mu\right)\mathrm{d}_{\rm{spec}}\mu.
\end{align}
We apply the Poisson summation formula for the $n$-sum, getting
\begin{align}\label{n sum after Poisson}
&\sum_{n=1}^\infty \frac{W_2(\frac{n}{\mathcal{N}})}{n^{1/2+\sigma+i\tau}}\tilde{S}(-\epsilon \ell_2, m, n; D, D\delta)
W_{\mu,N_1}\left(\frac{1}{2}+\sigma+i\tau, r^2mn\right)
K_{w_4}\left(\frac{\epsilon nm\ell_2}{D^2\delta};\mu\right)\nonumber\\
=&\,\sum_{a(\mod\delta)}\tilde{S}(-\epsilon \ell_2, m, a; D, D\delta)\sum_{n\equiv a(\mod \delta)}
\frac{W_2(\frac{n}{\mathcal{N}})}{n^{1/2+\sigma+i\tau}}
W_{\mu,N_1}\left(\frac{1}{2}+\sigma+i\tau, r^2mn\right)
K_{w_4}\left(\frac{\epsilon nm\ell_2}{D^2\delta};\mu\right)\nonumber\\
=&\,\,\frac{\mathcal{N}^{1/2-\sigma-i\tau}}{\delta}\sum_{a(\mod\delta)}
\tilde{S}(-\epsilon \ell_2, m, a; D, D\delta)
\sum_{n\in\mathbb{Z}}e\left(\frac{an}{\delta}\right)\nonumber\\
&\qquad\cdot
\int_{-\infty}^\infty \frac{W_2(x)}{x^{1/2+\sigma+i\tau}}
W_{\mu,N_1}\left(\frac{1}{2}+\sigma+i\tau, r^2mx\mathcal{N}\right)
K_{w_4}\left(\frac{\epsilon x\mathcal{N}m\ell_2}{D^2\delta};\mu\right)e\left(-\frac{n\mathcal{N}x}{\delta}\right)\dd x.
\end{align}
Inserting \eqref{n sum after Poisson} into \eqref{Phi star_1,w_4}, using Lemma \ref{K_w_4} and letting
\bna
\rho_1=\Im(\mu_1+\mu_2),\quad \rho_2=\Im(\mu_1-\mu_2),
\ena
we get the resulting $\rho$-integral involving $\tilde{K}$-function is
\begin{align}\label{rho-integral}
\int_{\mathbb{R}^2} h_{T,M}(\mu)&W_{\mu,N_1}\left(\frac{1}{2}+\sigma+i\tau, r^2mx\mathcal{N}\right)\nonumber\\
&\cdot\int_0^\infty \tilde{K}_{i\rho_2}(2\sqrt{u})\exp\left(i\rho_1\log\frac{\pi^3|y|}{u^{3/2}}\right)
\exp\left(\frac{2i\pi^3y}{u}\right)\frac{\mathrm{d}u}{u}\text{spec}(\mu)\mathrm{d}\rho,
\end{align}
where
\bna
y=\frac{\epsilon x\mathcal{N}m\ell_2}{D^2\delta},
\ena
and $\rho=(\rho_1,\rho_2)$.

We use the strategy in \cite[Lemma 8]{BB} to deal with it.
By using \eqref{property1 for h}, the definition of $W_{\mu,N_1}$, and integration of by parts, we get
that the $\rho_1$-integral is negligible unless
\bna
u\asymp|y|^{2/3}.
\ena
For the $u$-integral in \eqref{rho-integral}, by differentiating $k$ times with respect to $y$ and by using \eqref{J- bound},
we see that each differentiation produces a factor
\bna
T|y|^{-1}+u^{-1}\asymp(T+|y|^{1/3})|y|^{-1}.
\ena
We can truncate the integral involving $J^{-}$-function in the same way, and finally get
\bea\label{K_w_4 bound}
|y|^k\frac{\partial^k}{\partial y^k}K_{w_4}(y;\mu)\ll (T+|y|^{1/3})^k.
\eea

Now we go back to the $x$-integral in \eqref{n sum after Poisson}.
With the help of \eqref{K_w_4 bound}, the definition of $W_{\mu,N_1}$, and the bounds for $\mathcal{M}$ and $\mathcal{N}$,
we integrate by parts and show that the $x$-integral is negligible unless $n=0$.
By opening $\tilde{S}(-\epsilon \ell_2, m, a; D, D\delta)$
and computing the $a$-sum, we obtain
\begin{align*}
\sum_{a(\mod\delta)}&\tilde{S}(-\epsilon \ell_2, m, a; D, D\delta)\\
&=\sum_{a(\mod \delta)}\,\sum_{\substack {C_1(\mod D), \,C_2(\mod D\delta)\\(C_1,D)=(C_2,\delta)=1}}
e\left(m\frac{\bar{C_1}C_2}{D}+a\frac{\bar{C_2}}{\delta}-\epsilon \ell_2\frac{C_1}{D}\right)
=0,
\end{align*}
unless $\delta=1$.
With the help of this and the fact $m=\ell_1D$, we see that the contribution from $n=0$ to
\eqref{Phi star_1,w_4} is
\begin{align}\label{n=0 contribution}
&\mathcal{N}^{1/2-\sigma-i\tau}\ell_1^{-1/2-\sigma+i\tau}\sum_{\epsilon=\pm1}\sum_{r}\frac{1}{r^{1+2\sigma}}
\sum_{D=1}^\infty W_1\left(\frac{\ell_1D}{\mathcal{M}}\right)\frac{1}{D^{3/2+\sigma-i\tau}}
\sum_{\substack{C_1(\mod D)\\(C_1,D)=1}}
e\left(-\frac{\epsilon \ell_2C_1}{D}\right)\nonumber\\
&\cdot\int_{\Re(\mu)=0} h_{T,M}(\mu)\int_{-\infty}^\infty \frac{W_2(x)}{x^{1/2+\sigma+i\tau}}
W_{\mu,N_1}\left(\frac{1}{2}+\sigma+i\tau, x\mathcal{N}r^2\ell_1D\right)
K_{w_4}\left(\frac{\epsilon x\mathcal{N}\ell_1\ell_2}{D};\mu\right)\dd x\mathrm{d}_{\rm{spec}}\mu.
\end{align}
By inserting the definition \eqref{K_w_4 definition} of $K_{w_4}(y;\mu)$, the above double integrals become
\begin{align}\label{inserting Phi}
& \int_{\Re(\mu)=0} h_{T,M}(\mu)\nonumber\\
&\hskip 20pt\times\int_{-i\infty}^{i\infty}
\left(\int_{-\infty}^\infty x^{-u-1/2-\sigma-i\tau}W_2(x)
W_{\mu,N_1}\left(\frac{1}{2}+\sigma+i\tau, x\mathcal{N}r^2\ell_1D\right)\dd x\right)
\left|\frac{\ell_1\ell_2\mathcal{N}}{D}\right|^{-u}\nonumber\\
&\hskip 40pt\times\tilde{G}^{\epsilon}(u,\mu)\frac{\mathrm{d}u}{2\pi i}\mathrm{d}_{\rm{spec}}\mu.
\end{align}
By integration by parts with respect to the above $x$-integral,
we can restrict the $u$-integral to $|\tau+\Im(u)|\leq T^\varepsilon$.
Inserting the definition \eqref{tilde G definition} of $\tilde{G}^\epsilon(u,\mu)$ into \eqref{inserting Phi}, we obtain
the corresponding $\mu$-integral is
\begin{align*}
\int_{\Re(\mu)=0} h_{T,M}(\mu)&W_{\mu,N_1}\left(\frac{1}{2}+\sigma+i\tau, x\mathcal{N}r^2\ell_1D\right)\\
&\hskip 10pt\times\Biggl(\prod_{k=1}^3\frac{\Gamma(\frac{1}{2}(u-\mu_k))}{\Gamma(\frac{1}{2}(1-u+\mu_k))}
+i\epsilon\prod_{k=1}^3\frac{\Gamma(\frac{1}{2}(1+u-\mu_k))}{\Gamma(\frac{1}{2}(2-u+\mu_k))} \Biggr)\mathrm{d}_{\rm{spec}}\mu.
\end{align*}
We only need to consider the first part, since the second part is very similar.
For fixed $\sigma_1\in\mathbb{R}$, $t\in\mathbb{R}$ and $|t|\geq 1$, we have the Stirling formula
\begin{align*}
\Gamma(\sigma_1+it)=\sqrt{2\pi}e^{\pi|t|/2}|t|^{\sigma_1-1/2}e^{it(\log|t|-1)}
e^{i(\sigma_1-1/2)\lambda\pi/2} \left(1+O(|t|^{-1})\right),
\end{align*}
where $\lambda=1$, if $t\geq1$, and $\lambda=-1$, if $t\leq -1$.
For convenience, we denote $\mu_1=it_1$, $\mu_2=it_2$, $\mu_3=-i(t_1+t_2)=it_3$ and $\Im(u)=t$.
Here, in the essential integrated range, we have $|t_k|\asymp|t_k-t_k'|\asymp T$ for $1\leq k, k'\leq 3$ and $k\neq k'$,
since we are considering the generic case.
Without loss of generality, we assume $t_1>0$, $t_2>0$ and $t_3<0$, and get
\begin{equation}\label{use stirling}
\begin{split}
&\quad\ \prod_{k=1}^3\frac{\Gamma(\frac{1}{2}(u-\mu_k))}{\Gamma(\frac{1}{2}(1-u+\mu_k))} \\
&=2^{3/2}e^{i(-3t(\log 2+1)+\pi/4)}
\frac{e^{i(t-t_1)\log(t_1-t)+i(t-t_2)\log(t_2-t)+i(t_1+t_2+t)\log(t_1+t_2+t)}}
{(t_1-t)^{1/2}(t_2-t)^{1/2}(t_1+t_2+t)^{1/2}}
+O(T^{-5/2}).
\end{split}
\end{equation}
By the trivial estimate, the contribution from the error term in \eqref{use stirling} to \eqref{n=0 contribution}
is
\begin{align*}
&\ll\mathcal{N}^{1/2-\sigma}\ell_1^{-1/2-\sigma}\sum_{r}\frac{1}{r^{1+2\sigma}}
\sum_{D=1}^\infty W_1\left(\frac{\ell_1D}{\mathcal{M}}\right)\frac{1}{D^{3/2+\sigma}}
\bigg|\sum_{\substack{C_1(\mod D)\\(C_1,D)=1}}
e\left(-\frac{\epsilon \ell_2C_1}{D}\right)\bigg|\cdot T^{-5/2}\\
&\cdot\int_{\Re(u)=0\atop |\tau+\Im(u)|\leq T^\varepsilon}
\bigg|\int_{-\infty}^\infty x^{-u-1/2-\sigma-i\tau}W_2(x)
\int_{\Re(\mu)=0} h_{T,M}(\mu)
W_{\mu,N_1}\left(\frac{1}{2}+\sigma+i\tau, x\mathcal{N}r^2\ell_1D\right)\mathrm{d}_{\rm{spec}}\mu\dd x\bigg|\dd u\\
&\ll\mathcal{N}^{1/2-\sigma}\ell_1^{-1/2-\sigma}T^{-5/2}T^{3+\varepsilon}M^2
\sum_{D\asymp \frac{\mathcal{M}}{\ell_1}}\frac{(D,\ell_2)}{D^{3/2+\sigma}}\ll T^{2-3\sigma+\varepsilon}M^2.
\end{align*}
For the main term in \eqref{use stirling}, we denote
the phase function for the $\mu_1$-integral is
$$
\phi(t_1):=(t-t_1)\log(t_1-t)+(t_1+t_2+t)\log (t_1+t_2+t).
$$
Note that
$$
\phi'(t_1)=\log\frac{t_1+t_2+t}{t_1-t}\gg 1.
$$
And hence, by partial integration many times, the $\mu_1$-integral is negligible.
We finally prove that the contribution from the $w_4$-term in $\Sigma_{4}^{(1)}$ to \eqref{use Ku 1}
is $O(T^{2-3\sigma+\varepsilon}M^2)$, which is the second error term of Theorem \ref{MainThm}.
\subsection{The terms $\Sigma_{6}^{(1)}$ and $\Sigma_{6}^{(2)}$}
Similarly, we only consider the contribution from $w_6$-term in $\Sigma_{6}^{(1)}$.
By Lemma \ref{truncate 3rd}, we can truncate the sums at
\[
  \frac{(n\ell_2)^{1/3}(m\ell_1)^{1/6}}{D_1^{1/2}}\geq T^{1-\varepsilon},
  \quad \textrm{and} \quad
  \frac{(n\ell_2)^{1/6}(m\ell_1)^{1/3}}{D_2^{1/2}}\geq T^{1-\varepsilon},
\]
which means
$$
\frac{(mn\ell_1\ell_2)^{1/2}}{(D_1D_2)^{1/2}}\geq T^{2-\varepsilon}.
$$
This gives us
$$
D_1D_2\ll \frac{mn\ell_1\ell_2}{T^{4-\varepsilon}},
$$
which is impossible provided that $\ell_1$, $\ell_2\leq T^{1/3}$, since the essential sums of $n$ and $m$
are truncated at $mn\leq T^{3+\varepsilon}$.
Thus, the contribution of the $w_6$-term is $O(T^{-B})$.

\subsection{The terms $\mathcal{E}_{\max}^{(1)}$ and $\mathcal{E}_{\max}^{(2)}$}
\label{subsec:contribution from Eisenstein}
We only treat the contribution of the maximal Eisenstein series $\mathcal{E}_{\max}^{(1)}$,
since the other contribution can be handled similarly.
Firstly, we recall Weyl's law on $\rm GL_2$ (see Venkov \cite[Theorem 7.3]{Venkov} and Hejhal \cite[pg.\,119, Theorem 8.1]{Hejhal}),
\begin{equation}\label{Weyl law on GL(2)}
\sharp\left\{g: t_g\leq T\right\}=\frac{T^2}{12}-\frac{T\log T}{2\pi}+C_0T+O\left(\frac{T}{\log T}\right),
\end{equation}
where $C_0$ is a constant.
Next we use the subconvexity bounds for
$\rm GL(1)$ and $\rm GL(2)$ $L$-functions to handle the contribution from this part. More precisely, we denote
\begin{equation*}
\begin{split}
\Theta:=&\sum_{r}\frac{1}{r^{1+2\sigma}}\sum_{m,n}\frac{1}{(mn)^{1/2+\sigma}}\left(\frac{m}{n}\right)^{i\tau}
\sum_{g} \frac{1}{2\pi i} \int\limits_{\Re(\mu)=0}
\frac{h_1(\mu')}{\mathcal{N}_{\mu,g}}
\overline{B_{\mu,g}(n, m)} B_{\mu,g}(\ell_1, \ell_2) \dd\mu,
\end{split}
\end{equation*}
where $\mu'$ is defined in \eqref{definition mu'}. Inserting the definitions
\eqref{test h1} and \eqref{W_j(s,y)} into $h_1(\mu)$ and $W_{\mu}(s, y)$, respectively,
and using \eqref{max Eisenstein coefficient property} and \eqref{max Eisenstein L-function},
we get
\begin{equation*}
\begin{split}
\Theta=\frac{1}{(2\pi i)^2}\sum_{g}  &\int\limits_{\Re(\mu)=0}
\frac{h_{T,M}(\mu')}{\mathcal{N}_{\mu,g}}B_{\mu,g}(\ell_1, \ell_2)\int\limits_{(3)}\pi^{-3u}\,
\zeta(\overline{s}+u+2\mu) \zeta(s+u-2\mu)\\
&\times L(\overline{s}+u-\mu,g)\,L(s+u+\mu,g)\prod_{k=1}^{3} \frac{\Gamma\left(\frac{s+u-\mu'_{k}}{2}\right)}
{\Gamma\left(\frac{s-\mu'_{k}}{2}\right)}\frac{\Gamma\left(\frac{\overline{s}+u+\mu'_{k}}{2}\right)}
{\Gamma\left(\frac{\overline{s}+\mu'_{k}}{2}\right)}
G(u)\frac{\mathrm{d}u}{u}\mathrm{d}\mu,
\end{split}
\end{equation*}
where $s=1/2+\sigma+i\tau$.  By the definition of $h_{T,M}$, we have $h_{T,M}(\mu')$
is negligibly small unless
\begin{equation*}
\begin{split}
|\mu_k'-\mu_{0,k}|\leq M,\quad{for}\,\,k=1,2,3.
\end{split}
\end{equation*}
Recall that $|\mu_{0,k}|\asymp T$ and \eqref{N max}. We now shift the line of integration to $\Re(u)=\varepsilon$
and use Stirling's formula, the subconvexity bounds
(see Bourgain \cite[Theorem 5]{Bourgain}, and Meurman \cite{Meurman})
\begin{equation}\label{subconvexity bounds}
\begin{split}
\zeta(1/2+it)\ll(1+|t|)^{13/84+\varepsilon},\quad L(1/2+it,g)\ll(1+|t|)^{1/3+\varepsilon}
\end{split}
\end{equation}
and \cite[(5.20)]{IK} to get
\begin{equation*}
\begin{split}
\Theta\ll T^{41/42-41\sigma/21+\varepsilon}M(\ell_1\ell_2)^\vartheta\sum_{T-M\leq\mu_g\leq T+M}1.
\end{split}
\end{equation*}
Finally, using Weyl's law \eqref{Weyl law on GL(2)}, we conclude that
\begin{equation}\label{bound 8}
\begin{split}
\Theta\ll T^{83/42-41\sigma/21+\varepsilon}M^2(\ell_1\ell_2)^\vartheta,
\end{split}
\end{equation}
which gives the third error term of Theorem \ref{MainThm}.
\subsection{The terms $\mathcal{E}_{\min}^{(1)}$ and $\mathcal{E}_{\min}^{(2)}$}
\label{subsec:contribution from Eisenstein2}
Similarly, to treat the contribution of the minimal Eisenstein series, we define
\begin{equation*}
\begin{split}
\Xi:=&\sum_{r}\frac{1}{r^{1+2\sigma}}\sum_{m,n}\frac{1}{(mn)^{1/2+\sigma}}\left(\frac{m}{n}\right)^{i\tau}
\frac{1}{24(2\pi i)^2} \iint\limits_{\Re(\mu)=0} \frac{h_1(\mu)}{\mathcal{N}_{\mu}}
\overline{A_{\mu}(n, m)} A_{\mu}(\ell_1, \ell_2) \dd\mu_1 \dd\mu_2.
\end{split}
\end{equation*}
Inserting the definitions
\eqref{test h1} and \eqref{W_j(s,y)} into $h_1(\mu)$ and $W_{\mu}(s, y)$,
and using \eqref{min Eisenstein coefficient property} and \eqref{min Eisenstein L-function},
we get
\begin{equation*}
\begin{split}
\Xi=&\,\frac{1}{24(2\pi i)^2} \iint\limits_{\Re(\mu)=0}
\frac{h_{T,M}(\mu)}{\mathcal{N}_{\mu}}
A_{\mu}(\ell_1, \ell_2)\\
&\times \frac{1}{2\pi i}\int\limits_{(3)}\pi^{-3u}\,\prod_{k=1}^{3} \zeta(s+u+\mu_k)\zeta(\overline{s}+u-\mu_k)
\frac{\Gamma\left(\frac{s+u-\mu_{k}}{2}\right)}
{\Gamma\left(\frac{s-\mu_{k}}{2}\right)}\frac{\Gamma\left(\frac{\overline{s}+u+\mu_{k}}{2}\right)}
{\Gamma\left(\frac{\overline{s}+\mu_{k}}{2}\right)}
G(u)\frac{\mathrm{d}u}{u}\dd\mu_1 \dd\mu_2,
\end{split}
\end{equation*}
where $s=1/2+\sigma+i\tau$.
We shift the line of integration to $\Re(u)=\varepsilon$ and use Stirling's formula,
the subconvexity bound \eqref{subconvexity bounds},
and \cite[(5.20)]{IK} to conclude that
\begin{equation*}
\begin{split}
\Xi\ll T^{13/14-13\sigma/7+\varepsilon}M^2(\ell_1\ell_2)^\vartheta,
\end{split}
\end{equation*}
which is smaller than the contribution of the maximal Eisenstein series given in \eqref{bound 8}.
\section{The mollified second moment}\label{sec:mollified2ndmoment}
In this section, we will give the \textbf{proof of Proposition \ref{second moment}}.
Recall the definition of $M(s,\phi_j)$ from \eqref{mollifier definition}.
By applying the Hecke relation \eqref{Hecke relation}, we obtain
\[
\begin{split}
|M(1/2+\sigma+i\tau,\phi_j)|^2 &=\sum_{\ell_1}\sum_{\ell_2} \frac{A_j(1,\ell_2)\overline{A_j(1,\ell_1)}}{\ell_1^{1/2+\sigma-i\tau}\ell_2^{1/2+\sigma+i\tau}}
x_{\ell_1}x_{\ell_2} \\
&=\sum_{d}\sum_{(\ell_1,d)=1}\sum_{(\ell_2,d)=1} \frac{A_j(\ell_1,\ell_2)}{d^{1+2\sigma}\ell_1^{1/2+\sigma-i\tau}\ell_2^{1/2+\sigma+i\tau}}
x_{d\ell_1}x_{d\ell_2}.
\end{split}
\]
Using the above relation and Theorem \ref{MainThm}, the left-hand side of \eqref{Aim2} is transformed into
\[
\begin{split}
&\frac{1}{\mathcal{H}}\sum_{d}\sum_{(\ell_1,d)=1}\sum_{(\ell_2,d)=1} \frac{x_{d\ell_1}x_{d\ell_2}}{d^{1+2\sigma}\ell_1^{1/2+\sigma-i\tau}\ell_2^{1/2+\sigma+i\tau}}
 \sum_{j}\frac{h_{T,M}(\mu_j)}{\mathcal{N}_j}A_j(\ell_1,\ell_2)|L(1/2+\sigma+i\tau,\phi_j)|^2\\
=&\,\zeta(1+2\sigma)\sum_{d}\sum_{(\ell_1,d)=1}\sum_{(\ell_2,d)=1}
\frac{x_{d\ell_1}x_{d\ell_2}}{(d\ell_1\ell_2)^{1+2\sigma}}\\
&+\,\zeta(1-2\sigma)\sum_{d}\sum_{(\ell_1,d)=1}\sum_{(\ell_2,d)=1}\frac{x_{d\ell_1}x_{d\ell_2}}{d^{1+2\sigma}\ell_1\ell_2}
\frac{1}{\mathcal{H}}\frac{1}{192\pi^5}
\int_{\Re(\mu)=0} h_{T,M}(\mu)\prod_{k=1}^{3}\left(-\frac{\mu^2_k}{4\pi^2}\right)^{-\sigma}\mathrm{d}_{\rm{spec}}\mu\\
&+O\bigg(\sum_{d}\sum_{\ell_1}\sum_{\ell_2} \frac{|x_{d\ell_1}x_{d\ell_2}|}{(d^2\ell_1\ell_2)^{1/2+\sigma}}
\,\big((\ell_1\ell_2)^{-1/2+\sigma}T^{-3/4-6\sigma}+T^{-1-3\sigma+\varepsilon}
+T^{-43/42-41\sigma/21+\varepsilon}(\ell_1\ell_2)^\vartheta\big)\bigg)\\
:=&\,\mathcal{S}_1+\mathcal{S}_2+\mathcal{E}_1+\mathcal{E}_2+\mathcal{E}_3.
\end{split}
\]
Next, we apply a strategy similar to that of Goldfeld--Huang \cite[\S\,4]{GH} to deal with the terms $\mathcal{S}_1$ and $\mathcal{S}_2$.
\subsection{The treatment of $\mathcal{S}_1$}
By the integral definition \eqref{x_l definition} of $x_{d\ell_i}$, ($i=1,2$), we have
\begin{equation}\label{S_1}
\begin{split}
\mathcal{S}_1=&\zeta(1+2\sigma)
\frac{1}{(2\pi i)^2}\int_{(3)}\int_{(3)}G_1(u,v;\sigma)\frac{L^{2u}-L^u}{u^2}\frac{L^{2v}-L^v}{v^2}
\frac{\mathrm{d} u\mathrm{d}v}{(\log L)^2}\\
\end{split}
\end{equation}
where
\begin{equation*}
\begin{split}
G_1(u,v;\sigma)=&\,\sum_{d}\frac{\mu^2(d)}{d^{1+u+v+2\sigma}}\sum_{(\ell_1,d)=1}\sum_{(\ell_2,d)=1}
\frac{\mu(\ell_1)\mu(\ell_2)}{\ell_1^{1+u+2\sigma}\ell_2^{1+v+2\sigma}}\\
=&\,\prod_p\big(1-p^{-1-u-2\sigma}-p^{-1-v-2\sigma}+p^{-1-u-v-2\sigma}+p^{-2-u-v-4\sigma}\big)\\
=&\,\frac{\zeta(1+u+v+2\sigma)}{\zeta(1+u+2\sigma)\zeta(1+v+2\sigma)}H_1(u,v;\sigma).
\end{split}
\end{equation*}
The Euler product defining $H_1(u,v;\sigma)$ converges absolutely in the region
\bna
\Re(u)+2\sigma>-1/2,\quad\Re(v)+2\sigma>-1/2,\quad\Re(u)+\Re(v)+2\sigma>-1/2.
\ena
\begin{remark}
Recall that in the region $\Re(s)\geq1-c/\log(|\Im(s)|+3)$ (here $c$ is some positive constant),
$\zeta(s)$ is analytic except for a single pole at $s=1$, and has no zeros and satisfies
(see e.g., \cite[(3.11.7) and (3.11.8)]{Titchmarsh})
\bna
\zeta(s)^{-1}\ll \log(|\Im (s)|+3),\quad\text{and}\quad
\zeta'(s)/\zeta(s)\ll \log(|\Im (s)|+3).
\ena
\end{remark}
To evaluate the integral in \eqref{S_1}, we shift both contours to the line $\Re(u)=\Re(v)=1/\log T$ without passing any poles and truncate the $v$-integral at
$|\Im(v)|\leq T$ with an error of $O(T^{-1+\varepsilon})$. Now, we have
\begin{equation}\label{New S_1}
\begin{split}
\mathcal{S}_1=\frac{\zeta(1+2\sigma)}{(2\pi i)^2}\int_{\big(\frac{1}{\log T}\big)}\int_{\frac{1}{\log T}-iT}^{\frac{1}{\log T}+iT}
&\,\frac{\zeta(1+u+v+2\sigma)}{\zeta(1+u+2\sigma)\zeta(1+v+2\sigma)}H_1(u,v;\sigma)\\
&\qquad\qquad\cdot\frac{L^{2u}-L^u}{u^2}\frac{L^{2v}-L^v}{v^2}
\frac{\mathrm{d} u\mathrm{d}v}{(\log L)^2}+O(T^{-1+\varepsilon}).
\end{split}
\end{equation}
Shifting the $u$-integral to $\mathcal{C}_1$ given by
\bna
\mathcal{C}_1:=\bigg\{u: \Re(u)=-2\sigma-\frac{c}{\log^{3/4}(2+|\Im(u)|)}\bigg\},
\ena
where $c>0$ is a suitable constant such that $\zeta(s)$ has no zero in the region $\Re(s)>1-c\log^{-3/4}(2+|\Im(s)|)$ and
$c\log^{-3/4}(2)<1/3$,
we encounter two simple poles at $u=0$ and $u=-2\sigma-v$. The first pole yields a residue
\begin{equation}\label{main term S_1}
\begin{split}
&\frac{1}{2\pi i}\int_{\frac{1}{\log T}-iT}^{\frac{1}{\log T}+iT}
H_1(0,v;\sigma)\frac{L^{2v}-L^v}{v^2}\frac{\mathrm{d}v}{\log L}
=1+O(T^{-\delta/2-2\delta\sigma+\delta\varepsilon}).
\end{split}
\end{equation}
Here we extend the integration to the full line with the error term $O\big(T^{-1+\varepsilon}\big)$,
and shift the contour to $\Re(v)=-1/2-2\sigma+\varepsilon$.
The second pole yields a residue
\begin{equation}\label{residue 4}
\begin{split}
&\frac{\zeta(1+2\sigma)}{(\log L)^2}
\frac{1}{2\pi i}\int_{\frac{1}{\log T}-iT}^{\frac{1}{\log T}+iT}
\frac{L^{-4\sigma}-L^{-v-4\sigma}+L^{-2\sigma}
-L^{v-2\sigma}}{\zeta(1+v+2\sigma)\zeta(1-v)}H_1(-v-2\sigma,v;\sigma)
\frac{\mathrm{d}v}{(v+2\sigma)^2v^2}.
\end{split}
\end{equation}
Note that the contributions from the first three terms, that is,
$\frac{L^{-4\sigma}-L^{-v-4\sigma}+L^{-2\sigma}}{(v+2\sigma)^2v^2}$, are small.
Indeed, by shifting the integration to the contour $\mathcal{C}_{11}$ where
\begin{equation*}
\begin{split}
\mathcal{C}_{11}:=&\left[\frac{1}{\log T}-iT, \frac{1}{\log T}-i\right]\cup
\left[\frac{1}{\log T}-i, \frac{1}{4}-i\right]\cup\left[\frac{1}{4}-i, \frac{1}{4}+i\right]\\
&\cup\left[\frac{1}{4}+i, \frac{1}{\log T}+i\right]\cup\left[\frac{1}{\log T}+i, \frac{1}{\log T}+iT\right],
\end{split}
\end{equation*}
and using the standard bounds for the Riemann zeta function, one can show that the contributions from the the first three terms
are bounded by
\begin{equation}\label{bound 9}
\begin{split}
O\left(\frac{T^{-2\delta\sigma}}{\log^2T}\right).
\end{split}
\end{equation}
Now we consider the contribution from $\frac{L^{v-2\sigma}}{(v+2\sigma)^2v^2}$ in \eqref{residue 4}.
We move the line of integration to the contour $\mathcal{C}_{12}$ where
\begin{equation*}
\begin{split}
\mathcal{C}_{12}:=&\left[\frac{1}{\log T}-iT, -2\sigma-\frac{c}{\log^{3/4}T}-iT\right]\\ &\cup
\left[-2\sigma-\frac{c}{\log^{3/4}T}-iT, -2\sigma-\frac{c}{\log^{3/4}T}+iT\right]\\
&\cup\left[-2\sigma-\frac{c}{\log^{3/4}T}+iT, \frac{1}{\log T}+iT\right],
\end{split}
\end{equation*}
for the same constant $c$ in the definition of $\mathcal{C}_1$,
and pick up contributions from two simple poles at $v=0$ and $v=-2\sigma$.
The pole at $v=0$ contributes
\begin{equation}\label{residue 2}
\begin{split}
\frac{L^{-2\sigma}}{(2\sigma\log L)^2}H_1(-2\sigma,0;\sigma).
\end{split}
\end{equation}
The pole at $v=-2\sigma$ contributes
\begin{equation}\label{residue 3}
\begin{split}
-\frac{L^{-4\sigma}}{(2\sigma\log L)^2}H_1(0,-2\sigma;\sigma).
\end{split}
\end{equation}
Moreover, we find the $v$-integral on $\mathcal{C}_{12}$ can be bounded by
\begin{equation}\label{bound 3}
\begin{split}
O\left(T^{-4\delta\sigma}e^{-c'\log^{1/4}T}\right),\quad\text{for some}\,\,c'>0.
\end{split}
\end{equation}
The remaining integral, for $u$ of \eqref{New S_1} on $\mathcal{C}_1$, is bounded by
\begin{equation}\label{bound 4}
\begin{split}
O\left(T^{-2\delta\sigma}e^{-c'\log^{1/4}T}\right),\quad\text{for some}\,\,c'>0,
\end{split}
\end{equation}
in view of standard bounds for $\zeta(s)$ in the zero-free region.
Combining \eqref{main term S_1} and \eqref{bound 9}--\eqref{bound 4}, we conclude that
\begin{equation}\label{S_1 contribution}
\begin{split}
\mathcal{S}_1=1+O(T^{-2\delta\sigma}).
\end{split}
\end{equation}

\subsection{The treatment of $\mathcal{S}_2$}
Recall that
\begin{equation*}
\begin{split}
\mathcal{S}_2=\zeta(1-2\sigma)\sum_{d}\sum_{(\ell_1,d)=1}\sum_{(\ell_2,d)=1}\frac{x_{d\ell_1}x_{d\ell_2}}{d^{1+2\sigma}\ell_1\ell_2}
\frac{1}{\mathcal{H}}\frac{1}{192\pi^5}
\int_{\Re(\mu)=0} h_{T,M}(\mu)\prod_{k=1}^{3}\left(-\frac{\mu^2_k}{4\pi^2}\right)^{-\sigma}\mathrm{d}_{\rm{spec}}\mu.
\end{split}
\end{equation*}
Using the trivial estimates, we obtain
\begin{equation}\label{S_2 trivial bound}
\begin{split}
\mathcal{S}_2\ll T^{-6\sigma}\sum_{d \leq L^2}\,\sum_{\ell_1\leq L^2/d}\,\sum_{\ell_2\leq L^2/d}
\frac{1}{d^{1+2\sigma}\ell_1\ell_2}\ll T^{-6\sigma}\log^2T.
\end{split}
\end{equation}
Next we focus on getting a better bound for $\mathcal{S}_2$, and define
\begin{equation*}
\begin{split}
\Delta:=\sum_{d}\sum_{(\ell_1,d)=1}\sum_{(\ell_2,d)=1}\frac{x_{d\ell_1}x_{d\ell_2}}{d^{1+2\sigma}\ell_1\ell_2}.
\end{split}
\end{equation*}
By the definition \eqref{x_l definition} of $x_{d\ell_i}$, ($i=1,2$), we have
\begin{equation}\label{Delta}
\begin{split}
\Delta=&\frac{1}{(2\pi i)^2}\int_{(3)}\int_{(3)}G_2(u,v;\sigma)\frac{L^{2u}-L^u}{u^2}\frac{L^{2v}-L^v}{v^2}
\frac{\mathrm{d} u\mathrm{d}v}{(\log L)^2}\\
\end{split}
\end{equation}
where
\begin{equation*}
\begin{split}
G_2(u,v;\sigma)=&\,\sum_{d}\frac{\mu^2(d)}{d^{1+u+v+2\sigma}}\sum_{(\ell_1,d)=1}\sum_{(\ell_2,d)=1}
\frac{\mu(\ell_1)\mu(\ell_2)}{\ell_1^{1+u}\ell_2^{1+v}}\\
=&\,\prod_p\big(1-p^{-1-u}-p^{-1-v}+p^{-1-u-v-2\sigma}+p^{-2-u-v}\big)\\
=&\,\frac{\zeta(1+u+v+2\sigma)}{\zeta(1+u)\zeta(1+v)}H_2(u,v;\sigma).
\end{split}
\end{equation*}
The Euler product defining $H_2(u,v;\sigma)$ converges absolutely in the region
\bna
\min\big\{\Re(u)+\Re(v)+2\sigma,\,\Re(u),\,\Re(v)\big\}>-1/2.
\ena
We may assume that $\sigma<\frac{100\log\log T}{\log T}$ since otherwise $T^{-2\sigma}<T^{-\sigma}/(\log T)^{100}$
and the integral may be bounded directly by using standard bounds for $\zeta(s)$ and $\zeta^{-1}(s)$ to the right of the $1$-line
(see also the trivial bound in \eqref{S_2 trivial bound}).
To study \eqref{Delta}, we shift the $u$, $v$-integrals to the lines $\Re(u)=\Re(v)=1/\log T$
without passing any poles and truncate the $v$-integral at
$|\Im(v)|\leq T$ with the error term $O(T^{-1+\varepsilon})$.
Now we have
\begin{equation}\label{new Delta}
\begin{split}
\Delta=\frac{1}{(2\pi i)^2}\int_{\big(\frac{1}{\log T}\big)}\int_{\frac{1}{\log T}-iT}^{\frac{1}{\log T}+iT}
&\,\frac{\zeta(1+u+v+2\sigma)}{\zeta(1+u)\zeta(1+v)}H_2(u,v;\sigma)\\
&\qquad\cdot\frac{L^{2u}-L^u}{u^2}\frac{L^{2v}-L^v}{v^2}
\frac{\mathrm{d} u\mathrm{d}v}{(\log L)^2}+O(T^{-1+\varepsilon}).
\end{split}
\end{equation}
Shifting the $u$-integral to $\mathcal{C}_2$ given by
\begin{equation*}\label{Contour C_2}
\begin{split}
\mathcal{C}_2:=\bigg\{u: \Re(u)=-\frac{c}{\log^{3/4}(2+|\Im(u)|)}\bigg\},
\end{split}
\end{equation*}
for the same constant $c$ in the definition of $\mathcal{C}_1$, we encounter a pole at $u=-v-2\sigma$, which yields a residue
\begin{equation}\label{residue 5}
\begin{split}
&\frac{1}{(\log L)^2}
\frac{1}{2\pi i}\int_{\frac{1}{\log T}-iT}^{\frac{1}{\log T}+iT}
\frac{L^{-4\sigma}-L^{-v-4\sigma}+L^{-2\sigma}-L^{v-2\sigma}}{\zeta(1-v-2\sigma)\zeta(1+v)}H_2(-v-2\sigma,v;\sigma)
\frac{\mathrm{d}v}{(v+2\sigma)^2v^2}.
\end{split}
\end{equation}
Note that the contributions from the first three terms, that is,
$\frac{L^{-4\sigma}-L^{-v-4\sigma}+L^{-2\sigma}}{(v+2\sigma)^2v^2}$, are small.
Indeed, by shifting the integration to the contour $\mathcal{C}_{21}$ where
\begin{equation*}
\begin{split}
\mathcal{C}_{21}:=&\left[\frac{1}{\log T}-iT, \frac{1}{\log T}-i\right]\cup
\left[\frac{1}{\log T}-i, \frac{1}{4}-i\right]\cup\left[\frac{1}{4}-i, \frac{1}{4}+i\right]\\
&\cup\left[\frac{1}{4}+i, \frac{1}{\log T}+i\right]\cup\left[\frac{1}{\log T}+i, \frac{1}{\log T}+iT\right],
\end{split}
\end{equation*}
and using the standard bounds for the Riemann zeta function, one can show that the contributions from the the first three terms
are bounded by
\begin{equation}\label{bound 5}
\begin{split}
O\left(\frac{T^{-2\delta\sigma}}{\log^2T}\right).
\end{split}
\end{equation}
Now we consider the contribution from $\frac{L^{v-2\sigma}}{(v+2\sigma)^2v^2}$ in \eqref{residue 5}.
We move the line of integration in \eqref{residue 5} to the contour $\mathcal{C}_{22}$ where
\begin{equation*}
\begin{split}
\mathcal{C}_{22}:=&\left[\frac{1}{\log T}-iT, -\frac{c}{\log^{3/4}T}-iT\right]\cup
\left[-\frac{c}{\log^{3/4}T}-iT, -\frac{c}{\log^{3/4}T}+iT\right]\\
&\cup\left[-\frac{c}{\log^{3/4}T}+iT, \frac{1}{\log T}+iT\right],
\end{split}
\end{equation*}
for the same constant $c$ in the definition of $\mathcal{C}_1$,
and pick up contributions from the poles at $v=0$ and $v=-2\sigma$.
The pole at $v=0$ contributes
\begin{equation}\label{residue 6}
\begin{split}
-\frac{L^{-2\sigma}}{\zeta(1-2\sigma)(2\sigma\log L)^2}H_2(-2\sigma,0;\sigma).
\end{split}
\end{equation}
The pole at $v=-2\sigma$ contributes
\begin{equation}\label{residue 7}
\begin{split}
\frac{L^{-4\sigma}}{\zeta(1-2\sigma)(2\sigma\log L)^2}H_2(0,-2\sigma;\sigma).
\end{split}
\end{equation}
Similarly, we find the integration in \eqref{residue 5} on $\mathcal{C}_{22}$ can be bounded by
\begin{equation}\label{bound 6}
\begin{split}
O\left(T^{-2\delta\sigma}e^{-c'\log^{1/4}T}\right),\quad\text{for some}\,\,c'>0.
\end{split}
\end{equation}
The remaining integral, for $u$ of \eqref{new Delta} on $\mathcal{C}_2$, is bounded by
\begin{equation}\label{bound 7}
\begin{split}
O\left(e^{-c'\log^{1/4}T}\right),\quad\text{for some}\,\,c'>0.
\end{split}
\end{equation}
Combining \eqref{new Delta} and \eqref{bound 5}--\eqref{bound 7}, and noting that
$H_2(-2\sigma,0;\sigma)=H_2(0,-2\sigma;\sigma)=1$, we obtain
\begin{equation*}
\begin{split}
\Delta=\frac{L^{-2\sigma}(-1+L^{-2\sigma})}{(2\sigma\log L)^2}
+O\left(\frac{T^{-2\delta\sigma}}{\log^2T}\right).
\end{split}
\end{equation*}
Further, we conclude that
\begin{equation}\label{S_2 asymptotic formula}
\begin{split}
\mathcal{S}_2\ll T^{-2\sigma\delta}.
\end{split}
\end{equation}
\subsection{The contribution of the error terms}
By the trivial estimates, we have
\begin{equation}\label{ET1 contribution}
\begin{split}
\mathcal{E}_1\ll T^{-3/4-6\sigma}\sum_{d \leq L^2}\,\sum_{\ell_1\leq L^2/d}\,\sum_{\ell_2\leq L^2/d}
\frac{1}{d^{1+2\sigma}\ell_1\ell_2}\ll T^{-3/4-6\sigma+\varepsilon},
\end{split}
\end{equation}

\begin{equation}\label{ET2 contribution}
\begin{split}
\mathcal{E}_2\ll T^{-1-3\sigma+\varepsilon}\sum_{d \leq L^2}\,\sum_{\ell_1\leq L^2/d}\,\sum_{\ell_2\leq L^2/d}
\frac{1}{d^{1+2\sigma}(\ell_1\ell_2)^{1/2+\sigma}}\ll T^{-1-3\sigma+(2-4\sigma)\delta+\varepsilon},
\end{split}
\end{equation}
and
\begin{equation}\label{ET3 contribution}
\begin{split}
\mathcal{E}_3&\ll T^{-43/42-41\sigma/21+\varepsilon}
\sum_{d \leq L^2}\,\sum_{\ell_1\leq L^2/d}\,\sum_{\ell_2\leq L^2/d}
\frac{(\ell_1\ell_2)^{7/64}}{d^{1+2\sigma}(\ell_1\ell_2)^{1/2+\sigma}}\\
&\ll T^{-43/42-41\sigma/21+39\delta/16-4\sigma\delta+\varepsilon}
\ll T^{-2\sigma\delta},
\end{split}
\end{equation}
provided for $\delta\leq 1/3$.
Thus, combining \eqref{S_1 contribution} and \eqref{S_2 asymptotic formula}--\eqref{ET3 contribution},
we complete the proof of Proposition \ref{second moment}.
\\ \
\section{Proof of Theorem \ref{zero-density estimate}}\label{sec:Proof of main theorem}
In the proof of Theorem \ref{zero-density estimate}, we shall need the following summation formula,
which is an application of the $\rm GL(3)$ Kuznetsov trace formula.
\begin{theorem}\label{application 1}
Let $P=m_1n_1m_2n_2\neq 0$. We have for any $\varepsilon>0$ that
\begin{equation*}
\begin{split}
&\sum_{j} \frac{h_{T,M}(\mu_{j})}{\mathcal{N}_j} \overline{A_{j}(m_1, m_2)} A_{j}(n_1, n_2)\\
&\quad=\delta_{n_1, m_1} \delta_{n_2, m_2}\mathcal{H}
+O\left((TP)^\varepsilon\big(TP^{1/2}+T^3+TM^2P^\vartheta\big)\right),
\end{split}
\end{equation*}
where $\mathcal{H}$ is defined in \eqref{mathcal{H}} and $\vartheta=7/64$ is defined in \eqref{vartheta}.
\end{theorem}
\begin{proof}
The proof follows closely the strategy of Buttcane--Zhou \cite[Theorem 3.3]{BZ} or Blomer \cite[Theorem 5]{Blomer1},
which is given here for completeness.
Firstly, Lemmas \ref{truncate 2nd} and \ref{truncate 3rd} imply that for any fixed $B>0$,
\begin{equation*}
\begin{split}
\Phi_{w_6}\!
\left(-\frac{\epsilon_2m_1n_2D_2}{D_1^2},-\frac{\epsilon_1 m_2n_1D_1}{ D_2^2}\right)\ll T^{-B},
\end{split}
\end{equation*}
unless $D_1D_2\ll PT^{-4}(TP)^\varepsilon$, and
\begin{equation*}
\begin{split}
\Phi_{w_4}\!\left(\frac{\epsilon m_1m_2n_2}{D_1 D_2} \right)\ll T^{-B},
\end{split}
\end{equation*}
unless $D_1D_2\ll PT^{-3}(TP)^\varepsilon$.
As mentioned as before, in order to use these lemmas, it is necessary to know the $D_1,\,D_2$-sums may be terminated at some
high power of $T$. At that point we may replace $h_{T,M}$ with a smooth, compactly supported test function satisfying
\eqref{property1 for h}.
In the current context,
\cite[Lemma 3]{Blomer1} becomes
\begin{equation*}
\begin{split}
\sum_{\epsilon_1,\epsilon_2=\pm 1} \sum_{D_1\leq X_1}\sum_{D_2\leq X_2}
\frac{\big|S(\epsilon_2n_2,\epsilon_1n_1,m_1,m_2;D_1, D_2)\big|}{D_1D_2}
\ll(X_1X_2)^{1/2+\varepsilon}(m_1n_2, m_2n_1)^\varepsilon,
\end{split}
\end{equation*}
as the proof goes through essentially unchanged.
So we use the upper bounds $\Phi(y)\ll T^{3+\varepsilon}$
from Lemma \ref{truncate 4th}, and obtain $\Sigma_6\ll P^{1/2}T(TP)^\varepsilon$. Note that the different
sharp of the summation region may be handled by dyadic partitions.
Similarly, Larsen's bound \eqref{Larsen bound} implies
\begin{equation*}
\begin{split}
\sum_{\epsilon=\pm 1} \sum_{\substack{D_2 \mid D_1\\  m_2 D_1=n_1 D_2^2\\ D_1D_2\ll X}}
\frac{\big|\tilde{S}(-\epsilon n_2, m_2, m_1; D_2, D_1)\big|}{D_1D_2}
\ll X^{\varepsilon}\sum_{\substack{n_1|m_2D_1\\ m_2D_1^2\ll n_1X}}
\frac{(m_1, D_1)}{D_1}\ll X^\varepsilon(m_2n_1)^\varepsilon,
\end{split}
\end{equation*}
So we have $\Sigma_4\ll T^3(TP)^\varepsilon$. The treatment of the $w_5$-term is identical.

As discussed in \S\,\ref{subsec:contribution from Eisenstein} and \S\,\ref{subsec:contribution from Eisenstein2},
here we only treat the contribution of the maximal Eisenstein series,
since the minimal Eisenstein series can be handled similarly and the contribution will be smaller.
Combining the known bounds on the $L$-functions and the $\rm GL(2)$ Weyl law
together with definitions of $h_{T,M}$, we see that the essential region of the integration
on $\mu$ is of length $\asymp M$, and the essential number of the sum of $\mu_g$ is of
size $\asymp TM$.
Thus we imply the trivial bound $\mathcal{E}_{\min}+\mathcal{E}_{\max}\ll TM^2P^\vartheta(TP)^\varepsilon$
and complete the proof of this theorem.
\end{proof}
We record here frequently used estimate
$$\mathcal{H}=\frac{1}{192\pi^5} \int_{\Re(\mu)= 0} h_{T,M}(\mu) \mathrm{d}_{\rm{spec}}\mu=CT^3M^2+O(T^2M^3),$$
for some constant $C>0$ (see Buttcane--Zhou \cite[pg.\,15]{BZ}). Taking $m_1=n_1=m_2=n_2=1$ in Theorem \ref{application 1} immediately gives
\begin{corollary}\label{corollary}
With the assumptions in Theorem \ref{application 1}. We have for small $\varepsilon>0$ that
\begin{equation*}
\begin{split}
&\sum_{j} \frac{h_{T,M}(\mu_{j})}{\mathcal{N}_j}
=CT^3M^2+O(T^2M^3+T^{3+\varepsilon}).
\end{split}
\end{equation*}
\end{corollary}
Now we give the \textbf{proof of Theorem \ref{zero-density estimate}}.
We will apply Lemma \ref{The argument principle} to $L(s,\phi_j)M(s,\phi_j)$ with box bounded
by $1/2+1/\log T\pm 2iH$ and $3/2\pm 2iH$.
For $2/\log T<\sigma<1/2$, we have
\begin{equation*}
\begin{split}
N(\sigma, H; \phi_j):=&\sum_{L(1/2+\beta+i\gamma,\phi_j)=0\atop\beta>\sigma,\,|\gamma|<H}1
\leq\sum_{L(1/2+\beta+i\gamma,\phi_j)=0\atop\beta>2/\log T,\,|\gamma|<H}1
\leq \log T\sum_{L(1/2+\beta+i\gamma,\phi_j)=0\atop\beta>1/\log T,\,|\gamma|<H}\left(\beta-\frac{1}{\log T}\right)\\
\leq&\, 4H(\log T) \sum_{L(1/2+\beta+i\gamma,\phi_j)=0\atop\beta>1/\log T,\,|\gamma|<H}\cos\left(\frac{\pi \gamma}{2H}\right)
\sinh\left(\frac{\pi(\beta-1/\log T)}{2H} \right).
\end{split}
\end{equation*}
Since the zeros of $L(s,\phi_j)$ are also zeros of the function $L(s,\phi_j)M(s,\phi_j)$,
we use Lemma \ref{The argument principle} with $\omega(s)=L(s,\phi_j)
M(s,\phi_j):=LM(s,\phi_j)$, $W_0=1/\log T$ and $W_1=1$, and get that
$N(\sigma, H; \phi_j)$ is
\begin{equation*}
\begin{split}
&\leq\,\log T\bigg(\int_{-H}^H\cos\left(\frac{\pi t}{2H}\right)
\log|LM(1/2+1/\log T+it,\phi_j)|\mathrm{d}t\\
&+\int_{\frac{1}{\log T}}^{1}\sinh\left(\frac{\pi\big(\alpha-1/\log T\big)}{2H}\right)
\log\left|LM\big(1/2+\alpha+iH,\phi_j\big)
LM\big(1/2+\alpha-iH,\phi_j\big)\right|\mathrm{d}\alpha\\
&-\Re\int_{-H}^H\cos\left(\frac{\pi (1-1/\log T+it)}{2iH}\right)\log LM(3/2+it,\phi_j)\mathrm{d}t\bigg).
\end{split}
\end{equation*}
Since always
\bna
\log|x|\leq \frac{|x|^2-1}{2},
\ena
and the definition \eqref{N(sigma,H)}, $N(\sigma, H)$ is
\begin{equation*}
\begin{split}
\leq&\,\frac{\log T}{2}\bigg\{\int_{-H}^H\cos\left(\frac{\pi t}{2H}\right)
\frac{1}{\mathcal{H}}\sum_{j}\frac{h_{T,M}(\mu_j)}{\mathcal{N}_j}\left(|LM(1/2+1/\log T+it,\phi_j)|^2-1\right)\mathrm{d}t\\
&+\int_{\frac{1}{\log T}}^{1}\sinh\left(\frac{\pi(\alpha-1/\log T)}{2H}\right)
\frac{1}{\mathcal{H}}\sum_{j}\frac{h_{T,M}(\mu_j)}{\mathcal{N}_j}\\
&\qquad\qquad\times\bigg(\left|LM\big(1/2+\alpha+iH,\phi_j\big)\right|^2
+\left|LM\big(1/2+\alpha-iH,\phi_j\big)\right|^2-2\bigg)\mathrm{d}\alpha\\
&-2\Re\int_{-H}^H\cos\left(\frac{\pi (1-1/\log T+it)}{2iH}\right)
\frac{1}{\mathcal{H}}\sum_{j}\frac{h_{T,M}(\mu_j)}{\mathcal{N}_j}\log LM(3/2+it,\phi_j)\mathrm{d}t\bigg\}.
\end{split}
\end{equation*}
Using Proposition \ref{second moment} and Corollary \ref{corollary}, for some sufficiently small $\delta,\,\theta_1>0$, we conclude that
\begin{equation*}
\begin{split}
N(\sigma, H)\ll &\,HT^{-\sigma\delta}\log T,
\end{split}
\end{equation*}
uniformly in $3/\log T<H<T^{\theta_1}$.

\bigskip
\noindent{\bf Acknowledgements}
The authors would like to thank Professor Sheng-Chi Liu for his valuable suggestions on this topic.
H. Wang would like to thank the Alfr\'{e}d R\'{e}nyi Institute of Mathematics for providing a great working environment
and the China Scholarship Council for its financial support.

\end{document}